\documentclass{amsart}

\usepackage{graphicx,color,hyperref,verbatim}
\usepackage{subcaption}
\numberwithin{equation}{section}

\usepackage[initials,nobysame]{amsrefs}
\usepackage{amsmath}
\usepackage{amssymb}
\usepackage{amsfonts}
\usepackage{amsthm}
\usepackage{bm}
\usepackage{bbm}
\usepackage{mathtools}
\usepackage{esint}
\usepackage{ascmac}
\usepackage{framed}
\usepackage{colonequals}
\usepackage{placeins}

\usepackage{tikz}
\usepackage{tikz-3dplot}
\usetikzlibrary{arrows.meta,calc,intersections,3d,patterns}
\usepackage{pgfplots}
\pgfplotsset{compat=1.16}

\definecolor{lightgray}{rgb}{0.75,0.75,0.75}
\definecolor{nemoto}{rgb}{0.27,0.75,1}
\definecolor{saki}{rgb}{0.3,0.38,1}
\definecolor{violet}{rgb}{0.58,0,0.83}
 {\endMakeFramed}

\newtheorem{theorem}{Theorem}[section]
\newtheorem{proposition}[theorem]{Proposition}
\newtheorem{corollary}[theorem]{Corollary}
\newtheorem{lemma}[theorem]{Lemma}
\theoremstyle{definition}

\newtheorem*{acknowledgements}{Acknowledgements}
\theoremstyle{remark}
\newtheorem{remark}[theorem]{Remark}

\newcommand{\R}{\mathbb{R}}
\newcommand{\Z}{\mathbb{Z}}
\renewcommand{\S}{\mathbb{S}}
\newcommand{\n}{\mathbf{n}}
\newcommand{\e}{\mathbf{e}}
\DeclareMathOperator{\supp}{supp}

\newcommand{\C}{\mathbb{C}}

\newcommand{\boD}{\mathcal{D}}

\newcommand{\boM}{\mathcal{M}}

\newcommand{\la}{\ensuremath{\lambda}}

\newcommand{\te}{\ensuremath{\theta}}

\renewcommand{\Im}{\mathop{{\rm Im}}\nolimits}

\renewcommand{\Re}{\mathop{{\rm Re}}\nolimits}

\DeclareMathOperator{\sign}{{\rm sign}}

\makeatletter												%
  \@addtoreset{equation}{section}						%
\makeatother

\newcommand{\nor}[2]{\left\| {#1} \right\|_{#2}}		
\newcommand{\ovl}[1]{\overline{#1}}					

\newcommand{\del}{{\delta}}								
\newcommand{\rd}{{\partial}}								
\newcommand{\nab}{{\nabla}}							

\newcommand{\bh}{{\mathbf{h}}}
\newcommand{\bom}{{\mathbf{m}}}
\newcommand{\bn}{{\mathbf{n}}}
\newcommand{\boe}{{\mathbf{e}}}
\newcommand{\bphi}{{\bm{\phi}}}

\title[Chiral magnetic fields of general degree]{Phase transition thresholds and chiral magnetic fields of general degree}

\author{Slim Ibrahim}
\address[S.~Ibrahim]{Department of Mathematics and Statistics, University of Victoria, BC-Canada \& RIMS, Kyoto University, Kyoto-Japan
}
\email{ibrahims@uvic.ca}

\author{Tatsuya Miura}
\address[T.~Miura]{Department of Mathematics, Graduate School of Science, Kyoto University, Kyoto, Japan}
\email{tatsuya.miura@math.kyoto-u.ac.jp}

\author{Carlos Rom\'{a}n}
\address[C.~Rom\'{a}n]{Facultad de Matem\'aticas e Instituto de Ingenier\'ia Matem\'atica y Computacional, Pontificia Universidad Cat\'olica de Chile, Vicu\~na Mackenna 4860, 7820436 Macul, Santiago, Chile}
\email{carlos.roman@uc.cl}

\author{Ikkei Shimizu}
\address[I.~Shimizu]{Department of Mathematics, Graduate School of Science, Kyoto University, Kyoto, Japan}
\email{shimizu.ikkei.8s@kyoto-u.ac.jp}

\date{\today}

\begin{document}

\begin{abstract}
We study a variational problem for the Landau--Lifshitz energy with Dzyaloshinskii--Moriya interactions arising in 2D micromagnetics, focusing on the Bogomol'nyi regime. We first determine the minimal energy for arbitrary topological degree, thereby revealing two types of phase transitions consistent with physical observations. In addition, we prove the uniqueness of the energy minimizer in degrees $0$ and $-1$, and nonexistence of minimizers for all other degrees. Finally, we show that the homogeneous state remains stable even beyond the threshold at which the skyrmion loses stability, and we uncover a new stability transition driven by the Zeeman energy.
\end{abstract}

\maketitle

\section{Introduction}

In this paper, we study the Landau--Lifshitz energy functional
\begin{equation}\label{eq:E_r}
    E_r[\bn] = D[\bn] + rH[\bn] + V[\bn],
\end{equation}
defined for sphere-valued maps $\bn:\R^2\to\S^2\subset\R^3$, where $r>0$ is a given constant.
Here the three terms of the energy are defined by
\begin{align*}
    D[\bn] &\colonequals \frac12 \int_{\R^2} |\nabla \bn|^2dx,\\
    H[\bn] &\colonequals \int_{\R^2} (\bn-\boe_3)\cdot(\nabla\times\bn)dx, \qquad \nabla\times\bn \colonequals 
    \begin{pmatrix}
        \partial_1\\
        \partial_2\\
        0
    \end{pmatrix}
    \times
        \begin{pmatrix}
        n_1\\
        n_2\\
        n_3
    \end{pmatrix},
    \\
    V[\bn] &\colonequals \frac18 \int_{\R^2} |\bn-\boe_3|^4dx = \frac{1}{2}\int_{\R^2}(1-n_3)^2dx,
\end{align*}
where $\{\boe_j\}_{j=1}^3\subset\R^3$ denotes the canonical basis.
We consider
\[
\boM \colonequals \{\bn:\R^2\to\S^2 \mid D[\bn] + V[\bn] < \infty \},
\]
the natural energy space equipped with the metric
\[
d_{\boM}(\bn,\bom)\colonequals \|\bn-\bom\|_{L^4(\R^2)}+\|\nabla(\bn-\bom)\|_{L^2(\R^2)}.
\]
All three terms in $E_r$ are finite on $\boM$ and depend continuously on the metric $d_{\boM}$, including the helicity term $H$ (see Section \ref{subsec:functional}).
We note that our model also encompasses the more general two-parameter family
\[
D[\bn] + \kappa H[\bn] + \mu V[\bn], \qquad \kappa\neq0,\ \mu > 0.
\]
Indeed, after the transformation $\bn(x)\mapsto\bn(\sign(\kappa)\sqrt{\mu}x)$, the energy reduces to the normalized form above with parameter $r = |\kappa| / \sqrt{\mu}$.

\medskip

The model \eqref{eq:E_r} arises in micromagnetics and describes the variational structure of thin chiral ferromagnetic films.  
The vector field $\bn$ represents the normalized magnetization and satisfies $|\bn|=1$.  
The Dirichlet energy $D$ encodes the Heisenberg exchange interaction, while the helicity term $H$ accounts for the Dzyaloshinskii--Moriya (DM) interaction generated by spin--orbit coupling in noncentrosymmetric crystals.  
The potential term $V$ arises as the critical combination $V = Z - A$ of the Zeeman energy $Z$ (due to an external magnetic field) and the magnetocrystalline anisotropy $A$:
\[
Z[\bn] \colonequals \int_{\R^2} (1 - n_3)\, dx,
\qquad
A[\bn] \colonequals \frac12 \int_{\R^2} (1 - n_3^2)\, dx.
\]

The DM interaction breaks spatial inversion symmetry and allows for the formation of localized, topologically nontrivial spin textures known as \emph{chiral magnetic skyrmions}.  
These are two-dimensional vortex-like solitons stabilized by the competition between exchange, DM, and anisotropy effects.  
Skyrmions were first predicted theoretically in the seminal works \cite{BogYab89,BogHub94,BogHub94vortex,BogHub1999}, 
and have since been experimentally observed in a variety of chiral magnets \cite{Muhlbauer2009,Yu2010,Yu2011,Leonov2016}. See also the review \cite{NagTok2013}. 
Their stability, nanoscale size, and controllability make them central objects of interest in both theoretical and applied micromagnetics \cite{Romming2013,Fert2013,Fert2017}.

\medskip

Since Melcher's pioneering work \cite{MR3269033}, there are a variety of mathematical studies of chiral magnetic skyrmions for the model \eqref{eq:E_r}, possibly with different potentials, see e.g.\ \cite{MR4489502,MR3639614,MR4179069,MR4105397,MR4204508,MR3853081,MR4961803,MR4945828,MR4091507,MR4677585,deng2025conformal,MR4198719,MR4061310,ibrahim2025global} and the references therein.
Our choice of the specific critical potential $V$, which follows \cite{MR3639614,MR4091507,MR4630481}, is motivated by the important factorization property based on the Bogomol'nyi trick (see, e.g., \cite[Proof of Lemma 2]{MR3639614}):
\begin{equation}\label{e1.1}
E_r [\bn] - 4\pi r^2 Q[\bn] = \frac {r^2}2 \int_{\R^2} |\boD_1^{r} \bn + \bn\times \boD^{r}_2 \bn|^2 dx + (1-r^2) D[\bn].
\end{equation}
Here $Q$ denotes the topological degree
\[
Q[\bn]\colonequals\frac{1}{4\pi}\int_{\R^2} \bn\cdot\partial_1\bn\times\partial_2\bn dx \in \Z,
\]
and $\boD_j^{r}\colonequals\partial_j-\frac{1}{r}\boe_j\times\cdot$ is the helical derivative for $j=1,2$.
As $Q$ is well defined and continuous on $\boM$ (see Section \ref{subsec:functional}), we can divide the set $\boM$ into a countable family of connected components as follows:
\begin{align*}
    \boM_k\colonequals\{ \bn\in\boM \mid Q[\bn]=k\},\qquad k\in\Z.
\end{align*}

The case of degree $Q=-1$ is particularly well studied (even with more general potentials) as this class is expected to contain isolated skyrmions.
Let
\begin{equation}\label{eq:hm}
\bh^r(x)\colonequals\bh\left(\frac{x}{2r}\right), \quad
\bh(x_1,x_2) \colonequals
\left(
\frac{-2x_2}{|x|^2 + 1} , \frac{2x_1}{|x|^2 +1} ,  \frac{|x|^2-1}{|x|^2+1}
\right).
\end{equation}
Note that $\bh^r$ is a harmonic map and $\bh^r\in\boM_{-1}$.
In fact, the map $\bh^r$ is a critical point of $E_r$ for all $r>0$.
D\"oring and Melcher \cite{MR3639614} proved that if $0<r\leq1$, then $\bh^r$ is a global minimizer of $E_r$ in $\boM_{-1}$,
\begin{align*}
    \inf_{\bn\in\boM_{-1}}E_r[\bn] = E_r[\bh^r] = 4\pi(1-2r^2).
\end{align*}
In particular, this implies that $\bh^r$ is a local minimizer in $\boM$; this supports that the map $\bh^r$ corresponds to an isolated skyrmion.
On the other hand, the first and last author \cite{MR4630481} showed that if $r>1$, then $\bh^r$ is not a local minimizer (unstable), and moreover $\inf_{\boM_{-1}}E_r=-\infty$.
This result rigorously verifies the phase transition, and how the stability of skyrmions is lost when the external force is weak.

In contrast, the case of general topological degree has been much less studied from a mathematical point of view; apart from Melcher's work \cite[Section 3]{MR3269033}, the authors are only aware of Muratov--Simon--Slastikov's recent work \cite{MR4961803} in a different setting (see Remark \ref{rem:comparison_MSS}).  
A broader consideration of arbitrary degree configurations, however, is meaningful for describing more complex pattern formations observed in the physics literature.

\medskip

In this paper we first determine the exact minimal energy for every topological degree $k\in\Z$, extending previous work for $k=-1$ \cite{MR3639614,MR4630481}.

\begin{theorem}\label{thm:main1_energy_formula}
    Let $k\in\Z$ and $r>0$.
    If $0<r\leq1$, then
    \begin{align*}
       \inf_{\bn\in\boM_k}E_r[\bn]=
       \begin{cases}
           4\pi|k|(1-2r^2) & (k<0),\\
           4\pi k & (k\geq0).
       \end{cases} 
    \end{align*}
    If $r>1$, then $\inf_{\bn\in\boM_k}E_r[\bn]=-\infty$ holds for each $k$.
\end{theorem}

\begin{figure}[htbp]
    \begin{center}\footnotesize
    \begin{tikzpicture}[scale=0.8]
        \draw [thick, -stealth](-0.5,0)--(4,0) node [anchor=north east]{$r$}; 
        \draw [thick, -stealth](0,-2.5)--(0,2.5) node [anchor=north east]{$\inf_{\boM_k} E_r$};
        \draw[thick,red] plot[domain=0:2.83, smooth,variable=\t] (\t,{subtract(0.5,0.125*\t*\t)});
        \draw[thick,blue] plot[domain=0:2.83, smooth,variable=\t] (\t,{subtract(1,0.25*\t*\t)});
        \draw[thick,teal] plot[domain=0:2.83, smooth,variable=\t] (\t,{subtract(1.5,0.375*\t*\t)});
        \node[anchor=east,red] at (0,0.5) {$k=-1$};
        \node[anchor=east,blue] at (0,1) {$k=-2$};
        \node[anchor=east,teal] at (0,1.5) {$k=-3$};
        \draw[orange,thick] 
        (2.828,0) -- (0,0) node[anchor=north east] {$k=0$};
        \draw[dashed] (2.83,2.5)--(2.83,-2.5) node[anchor=north] {$r=1$};
        \node[anchor=north east] at (2,0) {$r=\frac 1{\sqrt{2}}$};
    \end{tikzpicture}
    \begin{tikzpicture}[scale=0.8]
        \draw [thick, -stealth](-0.5,0)--(4,0) node [anchor=north east]{$r$}; 
        \draw [thick, -stealth](0,-2.5)--(0,2.5) node [anchor=north east]{$\inf_{\boM_k} E_r$};
        \node[anchor=east,red] at (0,0.5) {$k=1$};
        \node[anchor=east,blue] at (0,1) {$k=2$};
        \node[anchor=east,teal] at (0,1.5) {$k=3$};
        \draw[red,thick] (0,0.5) -- (2.83,0.5);
        \draw[blue,thick] (0,1) -- (2.832,1);
        \draw[teal,thick] (0,1.5) -- (2.834,1.5);
        \draw[dashed] (2.83,2.5)--(2.83,-2.5) node[anchor=north] {$r=1$};
    \end{tikzpicture}
    \end{center}
    \caption{Plots of the minimal energy for degrees $-3\leq k\leq3$.}
    \label{fig:minimal_energy}
\end{figure}
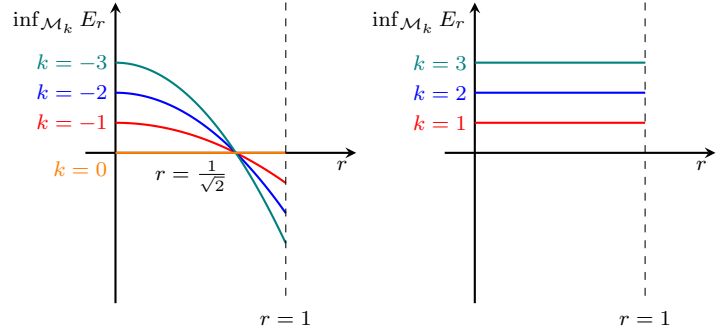

\begin{remark}[Two phase-transition thresholds]
    Theorem \ref{thm:main1_energy_formula} reveals two distinct phase-transition thresholds at $r=1/\sqrt{2}$ and $r=1$, cf.\ Figure \ref{fig:minimal_energy}:
    \begin{itemize}
        \item If $0<r<1/\sqrt{2}$, then 
        \[
        \inf_{k\neq0}\inf_{\boM_k}E_r > \inf_{\boM_0}E_r = \inf_{\boM}E_r=0,
        \]
        where the minimal energy is attained by the \emph{homogeneous state} $\bn\equiv \boe_3$.
        This regime thus corresponds to the \emph{homogeneous phase}.
        In addition, the single skyrmion $\bh^r$ is also a local minimizer, so it may also be regarded as the \emph{isolated skyrmion phase}, where individual skyrmions may coexist while remaining energetically independent.
        \item If $1/\sqrt{2}<r<1$, then the energy is bounded from below for each fixed topological degree $k$, but decreases to $-\infty$ as $k$ decreases,
        \[
        \dots>\inf_{\boM_1}E_r>\inf_{\boM_{0}}E_r>\inf_{\boM_{-1}}E_r>\dots \to \inf_{\boM}E_r = -\infty.
        \]
        We expect that this case corresponds to the \emph{skyrmion lattice phase}, where skyrmions form a periodic array, creating a two-dimensional lattice structure of skyrmions.
        This is consistent with the fact that a single skyrmion remains stable, while increasing their number reduces the total energy.
        \item Finally, if $r>1$, then even for a fixed degree $k\in\Z$, the energy becomes unbounded from below.
        In this regime, one can construct test configurations whose energy diverges to $-\infty$ by stretching in one direction.
        Thus we expect that this case corresponds to the \emph{helical phase}, characterized by one-dimensional helical spin configurations.
    \end{itemize}
    These phase transitions are well known from physical studies of chiral magnets (\cite{PhysRevB.89.094411,Gungordu2016}, see also \cite{BogHub94,BogHub94vortex,BogHub1999}). 
    Our result provides a rigorous verification of these transitions at a qualitative level, though it does not capture yet the detailed lattice or helical structure.
    The transition at $r=1/\sqrt{2}$ was also formally suggested in \cite{Ross2021}, on the basis of the observation that $E_r[\bh^r]$ becomes lower than $E_r[\boe_3]$ at this value.
\end{remark}

Next we address the existence and uniqueness of minimizers in a fixed-degree class $\boM_k$.
Although our framework allows for an explicit variational factorization and leads to exact energy formulae, it has certain subtleties concerning the existence of higher-degree minimizers.
By Theorem \ref{thm:main1_energy_formula}, clearly there is no minimizer if $r>1$, so we only analyze $0<r\leq1$.
Recall that there are minimizers for degree $k=0$ and $k=-1$, namely, the homogeneous state $\boe_3$ and the single skyrmion $\bh^r$, respectively.
Here we establish the rigidity theorem that they are the only possible minimizers (up to invariances), except at the endpoint case $r=1$.

\begin{theorem}\label{thm:main4_rigidity}
    Let $0<r<1$.
    Then the unique minimizer of $E_r$ in $\boM_0$ is $\boe_3$, and the unique minimizer of $E_r$ in $\boM_{-1}$ is $\bh^r$ up to translation; that is, $\bn(x)=\bh^r(x-x_0)$ for some $x_0\in\R^2$.
    Moreover, for any $k\in\Z\setminus\{0,-1\}$, the energy $E_r$ admits no minimizer in $\boM_k$.
\end{theorem}

This result suggests that, in order to investigate higher-degree configurations for this model, one needs to keep track of the precise quantitative behavior of minimizing sequences.
See also \cite{MR4732301,rupflin2023sharp} for quantitative rigidity results for harmonic maps of general degree.

The endpoint case $r=1$ is indeed a borderline case, allowing a nontrivial family of minimizers, see Appendix \ref{sec:couterexample_endpoint} for details.

\begin{remark}\label{rem:comparison_MSS}
Recently, Muratov--Simon--Slastikov \cite{MR4961803} established an abstract existence theorem for higher-degree minimizers for a closely related micromagnetic model.  
In their setting (translated into our conventions), the energy is defined on a bounded Lipschitz domain $\Omega \subset \R^2$ with Dirichlet boundary condition $\bn = \e_3$, and the potential $V$ is replaced by the purely anisotropic term $\mu A$ with $\mu \ge 0$.  
The DM interaction appears in a different form, but it is mathematically equivalent to ours up to a rotation and the addition of a divergence term.  
Within this framework, they proved the existence of minimizers for every degree $Q \le -1$ under explicit quantitative assumptions, requiring that $r$ is small and the domain $\Omega$ is large.
\end{remark}

Finally, we investigate the stability of critical points.  
The stability threshold of the skyrmion $\bn = \bh^r$ with $Q = -1$ has been previously established in \cite{MR4630481}.  
To the authors' knowledge, the only other known critical point of $E_r$ is the homogeneous state $\bn = \boe_3$ with $Q = 0$, whose stability is analyzed in our work.  
Remarkably, we find that the homogeneous state is strictly $L^2$-stable for all $r > 0$, even beyond the skyrmion stability threshold at $r = 1$.

\begin{theorem}\label{thm:main2_stability}
    Let $r>0$.  
    There exists $\delta=\delta(r)>0$ such that, for all $\bn\in \boM$ satisfying $0<\|\bn-\boe_3\|_{L^2(\R^2)}\leq \delta$, we have
    \[
        E_{r}[\bn] > E_{r}[\boe_3].
    \]
\end{theorem}

Our proof strategy is entirely different from that for the skyrmion, which relies on a second-variation analysis \cite{MR4630481}.  
Here we directly estimate the total energy to deduce the strict local minimality.

While the $L^2$-stability in Theorem \ref{thm:main2_stability} might be physically acceptable, it is also mathematically natural to ask for stability with respect to the canonical metric $d_\boM$.
However, this problem is subtle and remains open.

\medskip

To clarify the subtlety of the stability issue for the homogeneous state, we further add a perturbation by the Zeeman energy 
\begin{equation}\label{eq:E_rh}
    E_{r,h}\colonequals D+rH+V+hZ \qquad (h\in\R).
\end{equation}
That is, the potential $V=Z-A$ is perturbed to $V_h=(1+h)Z-A$. We show that our critical coupling $h=0$ is indeed a borderline case. 

A trivial difference is that, if $h\neq0$, then the energy space becomes
\begin{equation}\label{eq:space_M'}
    \boM'\colonequals\{\bn:\R^2\to\S^2 \mid D[\bn]+Z[\bn]<\infty \}\subset\boM,
\end{equation}
on which both $H[\bn]$ and $V[\bn]$ are still well defined (see Section~\ref{subsec:functional}), equipped with the metric
\[
d_{\boM'}(\bn,\bom)\colonequals\|\bn-\bom\|_{L^2(\R^2)}+\|\nabla(\bn-\bom)\|_{L^2(\R^2)}.
\]
Notice that $d_{\boM'}$ is stronger than $d_\boM$ 
due to Ladyzhenskaya's inequality
\begin{equation}\label{eq:Ladyzhenskaya}
    \|f\|_{L^4(\R^2)}\leq C\|f\|_{L^2(\R^2)}^{1/2}\|\nabla f\|_{L^2(\R^2)}^{1/2},
\end{equation}
which is a special case of the 2D Gagliardo--Nirenberg interpolation inequality.

More importantly, we discover that the stability changes at $h=0$ for the homogeneous state (as corollaries of slightly stronger stability and instability results; see Theorem~\ref{thm:h>0} and Theorem~\ref{thm:h<0} for details).

\begin{theorem}\label{thm:main3_staiblity_transition}
    Let $r>0$ and $h\neq0$.
    If $h>0$, then $\boe_3$ is a strict local minimizer of $E_{r,h}$ in $\boM'$.
    If $h<0$, then $\boe_3$ is not a local minimizer of $E_{r,h}$ in $\boM'$.
\end{theorem}

This paper is organized as follows: In Section \ref{sec:prelim} we recall some fundamental facts for our energy functionals and function spaces.
We then prove Theorem \ref{thm:main1_energy_formula} in Section \ref{sec:minimal_energy}, and Theorem \ref{thm:main4_rigidity} in Section \ref{sec:rigidity}.
Finally, in Section \ref{sec:stability} we prove Theorems \ref{thm:main2_stability} and \ref{thm:main3_staiblity_transition}.
The paper is complemented by Appendix \ref{sec:couterexample_endpoint} where we discuss details of counterexamples to Theorem \ref{thm:main4_rigidity} in the endpoint case $r=1$.


\begin{acknowledgements}
    SI acknowledges financial support from the Natural Sciences and Engineering Research Council of Canada (NSERC) under Grant No. 371637-2025. Part of this work was carried out during an extended visit to the Research Institute for Mathematical Sciences (RIMS). SI is grateful to colleagues and the staff at RIMS for their warm hospitality and stimulating research environment.
    TM is supported by JSPS KAKENHI Grant Numbers JP23H00085, JP23K20802, and JP24K00532.
    CR is supported by ANID FONDECYT 1231593. He wishes to thank the kind support and hospitality of the Tokyo Institute of Technology (currently the Institute of Science Tokyo) during the completion of part of this work.
    IS is supported by JSPS KAKENHI Grant Number JP23KJ1416.
\end{acknowledgements}


\section{Preliminaries}\label{sec:prelim}

\subsection{Properties of the functionals}\label{subsec:functional}
We first recall a few important properties of the 
functionals that we use.

First, recall that $H$ is well defined on $\boM$ in view of the inequality
\begin{equation}\label{eq:1124_01}
|(\bn-\boe_3)\cdot\nabla\times\bn| \leq C |\bn-\boe_3|^2|\nabla\bn| \in L^1(\R^2),
\end{equation}
which follows from $|(\bn-\boe_3)\cdot\nabla\times\bn|\leq c (1-n_3)|\nabla\bn|$, obtained by D\"oring--Melcher in \cite[Equation (8)]{MR3639614}, together with the simple identity for sphere-valued maps
\begin{equation}\label{eq:1124_02}
    2(1-n_3)=|\bn-\boe_3|^2.
\end{equation}
Then an integration by parts, whose validity is also ensured by D\"oring--Melcher \cite[Section 1.2]{MR3639614}, yields
\begin{align}
H[\bn]
&= 
\int_{\R^2} (n_1 \rd_2 n_3 - n_2\rd_1n_3) dx + \int_{\R^2} (n_3-1) (\rd_1n_2 -\rd_2n_1)dx \notag\\
&= 2 \int_{\R^2} (n_3-1) (\rd_1n_2-\rd_2n_1) dx. \label{eq:H_int_by_parts}
\end{align}
The last expression together with \eqref{eq:1124_02} further implies the continuity of $H$, and thus of $E_r$, on $\boM$.
By $\boM'\subset\boM$ and Ladyzhenskaya's inequality  \eqref{eq:Ladyzhenskaya}, each term of $E_{r,h}$ in \eqref{eq:E_rh} is also well defined and continuous on $\boM'$.

Next we discuss the topological degree $Q$ 
(cf.\ \cite{MR728866}).
It is easy to see that $Q$ is well defined  on $\boM \subset  L^\infty\cap \dot{H}^1$.
Moreover $Q$ is continuous on $\boM$, since any convergent sequence in $\boM$ has a subsequence that converges a.e.\ in $\R^2$, while if a sequence of maps $\{\bn_j\}_{j=1}^\infty\subset \dot{H}^1(\R^2;\S^2)$ satisfies
\begin{align*}
\bn_j
\to \bn \text{ a.e.\ in }\R^2,\qquad
\nab \bn_j \to \nab \bn \text{ in } L^2(\R^2)\quad 
\text{as } j\to\infty
\end{align*}
for some $\bn \in \dot{H}^1(\R^2;\S^2)$,
then $Q[\bn_j] \to Q[\bn]$ as $j\to\infty$. 
Indeed, triangle inequalities yield
\begin{align*}
&|Q[\bn_j] - Q[\bn]|\\
& \quad \leq
\left|\int_{\R^2} \bn_j \cdot\rd_1 \bn_j \times \rd_2(\bn_j-\bn) dx \right|
+
\left|
\int_{\R^2} \bn_j \cdot\rd_1 (\bn_j -\bn )\times \rd_2 \bn dx\right| \\
&\quad\qquad +
\left|
\int_{\R^2} (\bn_j-\bn) \cdot \rd_1 \bn\times \rd_2 \bn dx\right|.
\end{align*}
The first two terms are bounded by
$$
\nor{\rd_1 \bn_j}{L^2} \nor{\rd_2 (\bn_j-\bn)}{L^2} + \nor{\rd_2 \bn}{L^2} \nor{\rd_1 (\bn_j-\bn)}{L^2}
 \xrightarrow{j\to\infty} 0.
$$
The last term also converges to zero by the dominated convergence theorem (since $|\bn_j-\bn|\leq2$).
Finally, this continuity of $Q$ also implies that $Q[\bn]\in \Z$ for $\bn\in\boM$, as it is known that the set of maps $\bn:\R^2\to\S^2$ with $\bn-\boe_3\in C_c^\infty (\R^2;\R^3)$ is dense in $\boM$ (see the next paragraph) and all such maps satisfy $Q[\bn]\in\Z$ (see, e.g., \cite[Section 1.4]{MR488102}). 
Thus $Q$ is well defined and continuous on $\boM$, and hence also on $\boM'$.

Finally, for the reader's convenience, we briefly review how to derive the known density property
\begin{equation}\label{eq:smooth_density}
    \overline{\{ \bn: \R^2\to\S^2 \mid \bn-\boe_3 \in C_c^\infty(\R^2;\R^3)  \}}^{d_\boM} = \boM.
\end{equation}
We first use, for example, D\"oring--Melcher's radial cut-off argument \cite[Appendix 1]{MR3639614} to show that the set of (possibly discontinuous) maps 
\[
\boM_c\colonequals \{ \bn\in\boM \mid \text{$\bn-\boe_3$ has compact (measure-theory) support} \}
\]
is dense in $\boM$.
Then we can further show that $\boM_c\cap C(\R^2;\S^2)$ is also dense in $\boM$.
Indeed, using the averaging map $\bar\bn_\varepsilon(x)\colonequals \fint_{B_\varepsilon(x)}\bn(y)dy$ for $\bn\in\boM_c$, which is continuous but may not be sphere-valued, we can define
\[
\bn_\varepsilon\colonequals \frac{\bar\bn_\varepsilon}{|\bar\bn_\varepsilon|} \in \boM_c\cap C(\R^2;\S^2) \quad\text{for all $0<\varepsilon\ll1$}
\]
such that $\bn_\varepsilon\to \bn$ in $\boM$ as $\varepsilon\to0$.
Note that we may view $\bn_\varepsilon=\bn*\bar{\rho}_\varepsilon$ for $\bar{\rho}_\varepsilon\colonequals \frac1{|B_\varepsilon(0)|}\chi_{B_\varepsilon(0)}$.
A key, nontrivial observation (going back to Schoen--Uhlenbeck \cite[Proposition 4]{MR710054}; see also \cite[Lemma A.1]{MR728866}) is that
\begin{equation}\label{eq:norm_uniform_convergence}
    \|1-|\tilde{\bn}_\varepsilon|\|_{L^\infty} \to 0
\end{equation}
holds as $\varepsilon\to0$ (although $\tilde{\bn}_\varepsilon$ may not uniformly converge), which guarantees that the division by $|\tilde{\bn}_\varepsilon|$ causes no issue in the regularity nor convergence.
Convergence \eqref{eq:norm_uniform_convergence} follows since Poincar\'e's inequality
\[
\int_{B_{\varepsilon}(x)} |\bn(y)-\bar{\bn}_\varepsilon(x)|dy \leq C\varepsilon\int_{B_{\varepsilon}(x)}|\nabla\bn(y)|dy
\]
and the Cauchy--Schwarz inequality yield
\begin{align*}
    |1-|\bar{\bn}_\varepsilon(x)|| &\leq \fint_{B_{\varepsilon}(x)} |\bn(y)-\bar{\bn}_\varepsilon(x)|dy \leq C \Big( \int_{B_\varepsilon(x)}|\nabla\bn(y)|^2dy \Big).
\end{align*}
Now the density in \eqref{eq:smooth_density} follows easily since any map $\bn\in\boM_c\cap C(\R^2;\S^2)$ can be approximated by smooth maps $\bn_\varepsilon\colonequals \frac{\bn*\rho_\varepsilon}{|\bn*\rho_\varepsilon|}\in \boM_c\cap C^\infty(\R^2;\S^2)$ with the standard mollifier $\rho_\varepsilon$ so that $d_\boM(\bn_\varepsilon,\bn)\to0$ and $\|\bn_\varepsilon-\bn\|_\infty\to0$.

We also remark that this density property can be understood more broadly through the framework of VMO regularity, cf.\ \cite[Theorem 2.1]{MR2376670}.

\subsection{Equivariance}
We say that a map $\bn:\R^2\to\S^2\subset\R^3$ is \emph{equivariant} if $\bn(R_\theta x)=\hat{R}_{m\theta}\bn(x)$ holds for some $m\in\Z$ and any $\theta\in\R$ and $x\in\R^2$, where $R_\theta$ (resp.\ $\hat{R}_\theta$) denotes the rotation in $\R^2$ (resp.\ horizontal rotation in $\R^3$).

In terms of the polar coordinates $(\rho,\psi)$ of $\R^2$, an equivariant map $\bn$ is of the form
\begin{align*}
\bn(\rho,\psi) = 
\begin{pmatrix}
\cos\Phi(\psi) \sin \Theta (\rho)\\
\sin\Phi(\psi) \sin \Theta (\rho)\\
\cos \Theta (\rho)
\end{pmatrix}
\end{align*}
for some $\Phi(\psi)=m\psi+\psi_0$ with $m\in\Z$ and $\psi_0\in\R$, and 
$\Theta:[0,\infty) \to \R$.
Specifically in this paper, two cases of $\Phi$ appear: The case $\Phi=
\psi+\frac \pi2$,
which represents skyrmion configuration of the Bloch type, and the case $\Phi=
\pi-\psi$, which stands for anti-skyrmion (see Figure \ref{fig:skyrmion_antiskyrmion}).

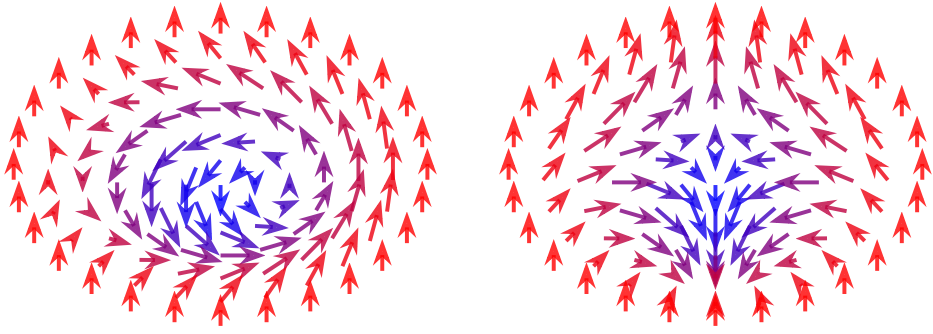
\begin{figure}[htbp]
  \centering

  \def\Rmax{5.5}
  \def\arrowlength{1.2}
  \def\el{45}  
  \def\az{0}   

  \begin{minipage}[t]{0.48\textwidth}
    \centering
    \tdplotsetmaincoords{\el}{\az}
    \begin{tikzpicture}[tdplot_main_coords,scale=0.5]

      \def\drawspin{%
        \pgfmathparse{180 * (1 - \r/\Rmax)}
        \pgfmathsetmacro{\thetadeg}{\pgfmathresult}

        \pgfmathparse{-sin(\thetadeg) * sin(\phideg)}
        \pgfmathsetmacro{\mx}{\pgfmathresult}
        \pgfmathparse{ sin(\thetadeg) * cos(\phideg)}
        \pgfmathsetmacro{\my}{\pgfmathresult}
        \pgfmathparse{ cos(\thetadeg)}
        \pgfmathsetmacro{\mz}{\pgfmathresult}

        \pgfmathparse{0.5*(\mz+1)}
        \pgfmathsetmacro{\normalizedmz}{\pgfmathresult}
        \pgfmathparse{1-\normalizedmz}
        \pgfmathsetmacro{\normalizedmmz}{\pgfmathresult}

        \pgfmathparse{0.1 * \mz}
        \pgfmathsetmacro{\zoffset}{\pgfmathresult}
        \coordinate (start) at (\x, \y, \zoffset);

        \draw[
          line width=1.5pt,
          opacity=0.8,
          draw={rgb,1:red,\normalizedmz; green,0; blue,\normalizedmmz},
          arrows={-Stealth[fill=white,length=10pt,width=7pt]}
        ]
          (start) -- ++(\arrowlength*\mx,
                        \arrowlength*\my,
                        \arrowlength*\mz);
      }

      \pgfmathsetmacro{\r}{0}
      \pgfmathsetmacro{\phideg}{0}
      \pgfmathsetmacro{\x}{0}
      \pgfmathsetmacro{\y}{0}
      \drawspin

      \pgfmathsetmacro{\r}{1*\Rmax/6}
      \foreach \k in {0,...,7}{%
        \pgfmathsetmacro{\phideg}{\k*360/8}
        \pgfmathsetmacro{\x}{\r*cos(\phideg)}
        \pgfmathsetmacro{\y}{\r*sin(\phideg)}
        \drawspin
      }

      \pgfmathsetmacro{\r}{2*\Rmax/6}
      \foreach \k in {0,...,11}{%
        \pgfmathsetmacro{\phideg}{\k*360/12}
        \pgfmathsetmacro{\x}{\r*cos(\phideg)}
        \pgfmathsetmacro{\y}{\r*sin(\phideg)}
        \drawspin
      }

      \pgfmathsetmacro{\r}{3*\Rmax/6}
      \foreach \k in {0,...,15}{%
        \pgfmathsetmacro{\phideg}{\k*360/16}
        \pgfmathsetmacro{\x}{\r*cos(\phideg)}
        \pgfmathsetmacro{\y}{\r*sin(\phideg)}
        \drawspin
      }

      \pgfmathsetmacro{\r}{4*\Rmax/6}
      \foreach \k in {0,...,19}{%
        \pgfmathsetmacro{\phideg}{\k*360/20}
        \pgfmathsetmacro{\x}{\r*cos(\phideg)}
        \pgfmathsetmacro{\y}{\r*sin(\phideg)}
        \drawspin
      }

      \pgfmathsetmacro{\r}{5*\Rmax/6}
      \foreach \k in {0,...,23}{%
        \pgfmathsetmacro{\phideg}{\k*360/24}
        \pgfmathsetmacro{\x}{\r*cos(\phideg)}
        \pgfmathsetmacro{\y}{\r*sin(\phideg)}
        \drawspin
      }

      \pgfmathsetmacro{\r}{6*\Rmax/6}
      \foreach \k in {0,...,27}{%
        \pgfmathsetmacro{\phideg}{\k*360/28}
        \pgfmathsetmacro{\x}{\r*cos(\phideg)}
        \pgfmathsetmacro{\y}{\r*sin(\phideg)}
        \drawspin
      }

    \end{tikzpicture}
  \end{minipage}
  \hfill
  \begin{minipage}[t]{0.48\textwidth}
    \centering
    \tdplotsetmaincoords{\el}{\az}
    \begin{tikzpicture}[tdplot_main_coords,scale=0.5]

      \def\drawspin{%
        \pgfmathparse{180 * (1 - \r/\Rmax)}
        \pgfmathsetmacro{\thetadeg}{\pgfmathresult}

        \pgfmathparse{-sin(\thetadeg) * cos(\phideg)}
        \pgfmathsetmacro{\mx}{\pgfmathresult}
        \pgfmathparse{ sin(\thetadeg) * sin(\phideg)}
        \pgfmathsetmacro{\my}{\pgfmathresult}
        \pgfmathparse{ cos(\thetadeg)}
        \pgfmathsetmacro{\mz}{\pgfmathresult}

        \pgfmathparse{0.5*(\mz+1)}
        \pgfmathsetmacro{\normalizedmz}{\pgfmathresult}
        \pgfmathparse{1-\normalizedmz}
        \pgfmathsetmacro{\normalizedmmz}{\pgfmathresult}

        \pgfmathparse{0.1 * \mz}
        \pgfmathsetmacro{\zoffset}{\pgfmathresult}
        \coordinate (start) at (\x, \y, \zoffset);

        \draw[
          line width=1.5pt,
          opacity=0.8,
          draw={rgb,1:red,\normalizedmz; green,0; blue,\normalizedmmz},
          arrows={-Stealth[fill=white,length=10pt,width=7pt]}
        ]
          (start) -- ++(\arrowlength*\mx,
                        \arrowlength*\my,
                        \arrowlength*\mz);
      }

      \pgfmathsetmacro{\r}{0}
      \pgfmathsetmacro{\phideg}{0}
      \pgfmathsetmacro{\x}{0}
      \pgfmathsetmacro{\y}{0}
      \drawspin

      \pgfmathsetmacro{\r}{1*\Rmax/6}
      \foreach \k in {0,...,7}{%
        \pgfmathsetmacro{\phideg}{\k*360/8}
        \pgfmathsetmacro{\x}{\r*cos(\phideg)}
        \pgfmathsetmacro{\y}{\r*sin(\phideg)}
        \drawspin
      }

      \pgfmathsetmacro{\r}{2*\Rmax/6}
      \foreach \k in {0,...,11}{%
        \pgfmathsetmacro{\phideg}{\k*360/12}
        \pgfmathsetmacro{\x}{\r*cos(\phideg)}
        \pgfmathsetmacro{\y}{\r*sin(\phideg)}
        \drawspin
      }

      \pgfmathsetmacro{\r}{3*\Rmax/6}
      \foreach \k in {0,...,15}{%
        \pgfmathsetmacro{\phideg}{\k*360/16}
        \pgfmathsetmacro{\x}{\r*cos(\phideg)}
        \pgfmathsetmacro{\y}{\r*sin(\phideg)}
        \drawspin
      }

      \pgfmathsetmacro{\r}{4*\Rmax/6}
      \foreach \k in {0,...,19}{%
        \pgfmathsetmacro{\phideg}{\k*360/20}
        \pgfmathsetmacro{\x}{\r*cos(\phideg)}
        \pgfmathsetmacro{\y}{\r*sin(\phideg)}
        \drawspin
      }

      \pgfmathsetmacro{\r}{5*\Rmax/6}
      \foreach \k in {0,...,23}{%
        \pgfmathsetmacro{\phideg}{\k*360/24}
        \pgfmathsetmacro{\x}{\r*cos(\phideg)}
        \pgfmathsetmacro{\y}{\r*sin(\phideg)}
        \drawspin
      }

      \pgfmathsetmacro{\r}{6*\Rmax/6}
      \foreach \k in {0,...,27}{%
        \pgfmathsetmacro{\phideg}{\k*360/28}
        \pgfmathsetmacro{\x}{\r*cos(\phideg)}
        \pgfmathsetmacro{\y}{\r*sin(\phideg)}
        \drawspin
      }

    \end{tikzpicture}
  \end{minipage}

  \caption{Schematic illustration of the skyrmion (left) and anti-skyrmion (right) profiles.}
  \label{fig:skyrmion_antiskyrmion}
\end{figure}

Notice that, under the equivariant symmetry,
$$
\bn\in \boM \quad \Longrightarrow \quad 
\Theta (0)\in \pi\Z \quad\text{and}\quad 
\Theta (\infty)\in 2\pi\Z.
$$
Indeed, direct computation yields
\begin{align*}
|\nab \bn|^2 = (\Theta')^2 + \frac{\sin^2 \Theta (\rho)}{\rho^2},&&
(1-n_3)^2 = (1-\cos \Theta (\rho))^2,
\end{align*}
which are integrable only if $\Theta (0)\in \pi\Z$ and $\Theta (\infty)\in 2\pi\Z$. 
For such a map, the degree is computed as
\begin{equation}\label{eq:Q}
\begin{aligned}
Q[\bn] 
&= \frac 1{4\pi} 
\int_0^{2\pi} \int_0^\infty \sin \Theta (\rho) \Theta'(\rho) \Phi'(\psi) d\rho d\psi \\
&= \frac 1{4\pi} [\cos \Theta(0) - \cos \Theta(\infty)] \cdot [\Phi(2\pi)- \Phi(0)],
\end{aligned}
\end{equation}
where $\Phi(2\pi) = \lim_{\del\to+0} \Phi(2\pi-\del)$.

The skyrmion $\bh^r$ defined in \eqref{eq:hm} is an equivariant map, represented as
$$
\bh^r =
\begin{pmatrix}
\cos (\psi+\frac \pi2) \sin \te^r(\rho)\\
\sin (\psi+\frac \pi2) \sin \te^r (\rho)\\
\cos \te^r(\rho)
\end{pmatrix}
,\qquad 
\sin \te^r (\rho) = \frac{4r\rho}{\rho^2+4r^2},
$$
with the boundary conditions $\te^r(0)=\pi$ and $\te^r(\infty)=0$. 
We write $\te\colonequals\te^1$ for 
simplicity. 
Note that $\bh^r\in \boM$ with $Q[\bh^r]=-1$, and that 
$\te^r$ satisfies 
$$
(\te^r)'(\rho) = -\frac{\sin \te^r(\rho)}{\rho}.
$$

For later use, we also introduce the anti-skyrmion \[
\tilde\bh^r(x)\colonequals\tilde\bh\left(\frac{x}{2r}\right), \quad
\tilde\bh(x_1,x_2) \colonequals
\left(
\frac{-2x_1}{|x|^2 + 1} , \frac{2x_2}{|x|^2 +1} ,  \frac{|x|^2-1}{|x|^2+1}
\right).
\]
In polar coordinates,
\begin{equation}\label{eq:defanti}
\tilde\bh^r = 
\begin{pmatrix}
\cos (\pi-\psi) \sin \te^r(\rho)\\
\sin (\pi -\psi) \sin \te^r (\rho)\\
\cos \te^r(\rho)
\end{pmatrix}
.
\end{equation}
Note that $\tilde\bh^r\in \boM$ and $Q[\tilde\bh^{r}]=1$.

\section{Minimal energy}\label{sec:minimal_energy}

In this section we prove Theorem \ref{thm:main1_energy_formula}, dividing the statement into three propositions (Propositions \ref{P2}, \ref{P4}, and \ref{P6}).
We first address the bounded-energy regime $0<r\leq1$ with negative degree (Proposition \ref{P2}), then nonnegative degree (Proposition \ref{P4}), and finally the unbounded-energy regime $r>1$ (Proposiion \ref{P6}).

\subsection{Negative degree case}

Our first goal is to prove the following

\begin{proposition}\label{P2}
Let $k\le -1$ be an integer. If $0<r\le 1$, we have
$$
\inf_{\bn\in\boM_{k}} E_r[\bn] = -4\pi k (1-2r^2) 
= |k|E_r[\bh^r].
$$
\end{proposition}
%
%
%
%
%

To prove Proposition \ref{P2}, and also for later use, we prepare Lemma \ref{L:cptHM} on approximations of (anti-)skyrmions by maps which are equal to $\boe_3$ outside a large ball.
We remark that, in fact, the general density property \eqref{eq:smooth_density} will suffice to compute the minimal energy in the bounded-energy regime (Propositions \ref{P2} and \ref{P4}).
However, in the unbounded-energy regime (Proposition \ref{P6}) we require additional control on the approximating sequence.
For this reason, we begin by constructing explicit approximations of the skyrmion and anti-skyrmion, and use them throughout this section.

\begin{lemma}\label{L:cptHM} 
There exists a family of $C^\infty$-functions $\te^r_R:[0,\infty)\to [0,\infty)$ for $R>0$ such that $\te^r_R(0)=\pi$ and $\supp \te^r_R \subset [0,R)$, and also such that for
\[
\bh^r_R\colonequals \begin{pmatrix}
\cos (\psi+\frac \pi2) \sin \te^r_R(\rho)\\
\sin (\psi+\frac \pi2) \sin \te^r_R(\rho)\\
\cos \te^r_R(\rho)
\end{pmatrix}, \quad
\tilde\bh^r_R \colonequals 
\begin{pmatrix}
\cos (\pi-\psi) \sin \te^r_R(\rho)\\
\sin (\pi -\psi) \sin \te^r_R (\rho)\\
\cos \te^r_R(\rho)
\end{pmatrix},
\]
the following properties hold:
\begin{enumerate}
    \item $\bh^r_R\in \boM_{-1}$ and $\tilde\bh^r_R \in \boM_1$ for all $R>0$;
    \item $\bh^r_R \to \bh^r$ and 
$\tilde\bh^r_R\to \tilde\bh^r$ in $\boM$ as $R\to\infty$;
    \item $\rd_{1} \bh^{r}_{R} (\cdot,0) \to 
\rd_{1} \bh^{r} (\cdot,0)$ in $L^2(\R)$  and $\bh^{r}_{R} (\cdot,0)-\boe_3 \to 
\bh^{r} (\cdot,0) -\boe_3$ in $L^4(\R)$ as $R\to\infty$.
\end{enumerate}
\end{lemma}
\begin{proof}
Let $\zeta_R:[0,\infty)\to [0,\infty)$ be a $C^\infty$-function such that
\begin{align*}
\zeta_R(\rho)=1 \text{ for } 0\le\rho\le \frac R4, &&
\zeta_R(\rho)=0 \text{ for } \rho\ge \frac R2, &&
\sup_{\rho\in[0,\infty)}|\zeta_R'| \le \frac CR, 
\end{align*}
with $C>0$ independent of $R$.
We show that the function $\te^r_R \colonequals \zeta_R \te^r$ satisfies the desired properties.
Clearly, we have $\te^r_R\in C^\infty_c([0,\infty))$ with $\te^r_R(0)=\pi$ and $\supp \te^r_R \subset [0,R)$, which also implies $\bh^r_R,\tilde\bh^r_R\in\boM$.
In view of \eqref{eq:Q} we also have $Q[\bh^r_R]=-1$ and $Q[\tilde\bh^r_R]=1$, thus confirming property (i).

For property (ii), we only argue the convergence of $\bh^r_R$, since the case of $\tilde\bh^r_R$ follows by a similar argument.
Then, it suffices to show
\begin{align}\label{eq:approx}
\nor{\nab( \bh^{r}_R -\bh^r) }{L^2}\to 0 \quad \text{and} \quad 
\nor{\bh^r_R - \bh^r}{L^4}\to 0
\end{align}
as $R\to\infty$. 
Since $\supp (\bh^r_R -\bh^r) \subset \R^2\setminus B_{\frac R4}(0)$ and $\supp (\bh^r_R-\boe_3) \subset B_{\frac R2}(0)$, we have
\begin{align*}
\nor{\nab (\bh^r_R -\bh^r)}{L^2(\R^2)}
\le \nor{\nab\bh^r}{L^2(\R^2\setminus B_{\frac R4}(0))} 
+ \nor{\nab\bh^r_R}{L^2(B_{\frac R2}(0)\setminus B_{\frac R4}(0))}.
\end{align*}
Since $\nab\bh^r\in L^2(\R^2)$, the first term clearly converges to $0$. 
For the second term, 
it follows that
\begin{align*}
\nor{\nab\bh^r_R}{L^2(B_{\frac R2}(0)\setminus B_{\frac R4}(0))}^2 
&= 2\pi \int_{R/4}^{R/2} 
\left([(\te^r_R)']^2 + \frac{\sin^2 \te^r_R}{\rho^2}\right) \rho d\rho \\
&= 2\pi \int_{R/4}^{R/2} 
\left([\zeta_R' \te^r + \zeta_R (\te^r)']^2 + \frac{\sin^2 (\zeta_R \te^r)}{\rho^2}\right) \rho d\rho\\
&\le 
C_r \int_{R/4}^{R/2} 
\left( \frac{(\te^r)^2}{R^2} +  
[(\te^r)']^2 + \frac{(\te^r)^2}{\rho^2}\right) \rho d\rho\\
&\le 
C_r \int_{R/4}^{R/2} 
\left( \frac{1}{\rho R^2} +  
\frac{1}{\rho^3}\right)  d\rho
\le \frac{C_r}{R^2} \xrightarrow{R\to\infty} 0.
\end{align*}
Similarly, we have
$$
\nor{\bh^r_R-\bh^r}{L^4(\R^2)}
\le \nor{\bh^r-\boe_3}{L^4(\R^2\setminus B_{\frac R4}(0))}+ \nor{\bh^r_R -\boe_3}{L^4(B_{\frac R2}(0)\setminus B_{\frac R4}(0))} .
$$
The first term converges to $0$ as $R\to\infty$ by $\bh^r -\boe_3\in L^4$. 
For the second term, we have
\begin{align*}
\nor{\bh^r_R -\boe_3}{L^4(B_{\frac R2}(0)\setminus B_{\frac R4}(0))}^4 
&= C \int_{R/4}^{R/2} (1- \cos \te^r_R)^2 \rho d\rho \\
&\le C \int_{R/4}^{R/2} (\zeta_R\te^r)^4 \rho d\rho \\
&\le C_r \int_{R/4}^{R/2} \frac{1}{\rho^3} d\rho 
\le \frac{C_r}{R^2} \xrightarrow{R\to\infty} 0. 
\end{align*}
Hence \eqref{eq:approx} follows. 
In particular, $\bh^r_R\in \boM$ for $R>0$, and $\bh^r_R\to \bh^r$ in $\boM$ as $R\to\infty$.

To prove property (iii), we may restrict the domain onto $\{x_1>0\}$ by symmetry. Then, we can write
\begin{align*}
\bh^r (x_1,0) = 
\begin{pmatrix}
0 \\ \sin \te^{r} (x_1) 
\\ \cos \te^r(x_1)
\end{pmatrix}
,&&
\bh^r_R (x_1,0) = 
\begin{pmatrix}
0 \\ \sin \te^{r}_R (x_1) 
\\ \cos \te^r_R(x_1)
\end{pmatrix}
.
\end{align*}
Hence similar estimates as above yield
$$
\nor{\rd_1\bh^r(\cdot ,0)- \rd_1\bh^r_R(\cdot ,0)}{L^2(x_1>0)}^2
+ 
\nor{\bh^r(\cdot ,0)- \bh^r_R(\cdot ,0)}{L^4(x_1>0)}^4 \le \frac{C_r}{R^3},
$$
which converges to zero as $R\to\infty$.
\end{proof}


%
%
\begin{proof}[Proof of Proposition \ref{P2}]
We first show that $-4\pi k (1-2r^2)$ is a lower bound of the energy. 
We recall the following topological lower bound:
\begin{equation}\label{e1.2}
D[\bn] \ge 4\pi |Q[\bn]|,
\end{equation}
which follows from the well-known identity
\begin{equation}\label{eq:factD}
D[\bn] = \frac 12\int_{\R^2} |\rd_1 \bn \mp \bn\times \rd_2 \bn\bm|^2 dx \pm 
4\pi Q[\bn].
\end{equation}
Since $0<r\le 1$ and $k<0$, 
the factorization \eqref{e1.1} and  \eqref{e1.2} yield
\begin{align*}
    E[\bn] &\ge (1-r^2) D[\bn] +4\pi r^2 Q[\bn]\\
    &\ge 
4\pi(1-r^2)  |Q[\bn]| +4\pi r^2 Q[\bn] = -4 \pi k (1-2r^2).
\end{align*}

Next we show the optimality of the above lower bound.
Let $R>0$, and denote by $\bh^r_{R}: \R^2\to \S^2$, the compactified harmonic map as in Lemma \ref{L:cptHM}. 
Define $\bn_R:\R^2\to\S^2$ to be a map created by gluing the constant map $\boe_3$ with $k$ compactified harmonic maps in disjoint regions, as in Figure \ref{fig:k-vortices}; more specifically,
\begin{equation}\label{eq:defnR}
\bn_R (x) = 
\begin{cases}
    \bh^r_{R} (x-a_j) & \text{if } x\in B_R(a_j),\ j=1,2,...,|k|,\\
    \boe_3 & \text{otherwise},
\end{cases}
\end{equation}
where $a_j\colonequals(10jR,0)\in\R^2$. 
Since $\bh^r_R-\boe_3$ is compactly supported in $B_{R}(0)$, 
the gluing is smooth and independent for each ball, so that 
$\bn_R-\boe_3\in C^\infty_c(\R^2:\R^3)$, 
which implies $\bn_R\in \boM$, 
and
$$
Q[\bn_R] = \frac {|k|}{4\pi}
\int_{B_{R}(0)} \bh^r_R \cdot \rd_1 \bh^r_R \times \rd_2 \bh^r_R dx
= 
(-k)\cdot Q[\bh^r_R] = k.
$$
Since $\bh^r_R\to \bh^r$ in $\boM$ as $R\to\infty$ by Lemma \ref{L:cptHM}, 
we obtain
$$
E[\bn_R] = |k| E[\bh^r_R] \xrightarrow{R\to\infty} 
|k| E[\bh^r] =-4\pi k (1-2r^2),
$$ 
completing the proof.
\end{proof}
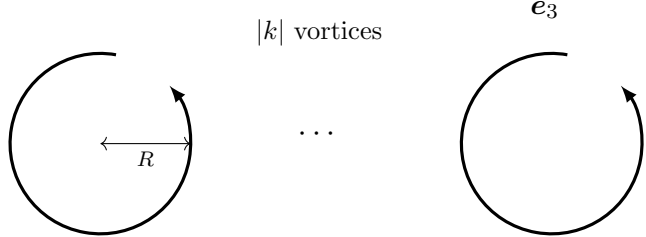
\begin{figure}[htbp]
    \footnotesize
    \begin{tikzpicture}[scale=1.5]
    \draw[-latex,very thick] (0,0) arc
    [
        start angle=80,
        end angle=400,
        x radius=0.8,
        y radius =0.8
    ] ;
    \draw[-latex,very thick] (4,0) arc
    [
        start angle=80,
        end angle=400,
        x radius=0.8,
        y radius =0.8
    ] ;
    \node at (1.8,-0.7) {{\Large $\cdots$}};
    \node at (1.8,0.2) {{\normalsize $|k|$ vortices}};
    \node at (3.8,0.4) {{\large $\bm{e}_3$}};
    \draw[<->] ({-0.8*cos(80)},{-0.8*sin(80)}) -- 
    ++(0:0.8);
    \node[anchor=north] at ({add(-0.8*cos(80),0.4)},{-0.8*sin(80)}) {$R$};
\end{tikzpicture}
\caption{The construction of $\bn_R$ in \eqref{eq:defnR}.}
\label{fig:k-vortices}
\end{figure}


\subsection{Nonnegative degree case}

Next we prove the following

\begin{proposition}\label{P4}
Let $k\geq0$ be an integer.
When $0<r\le 1$, we have
$$
\inf_{\bn\in \boM_k} E_r[\bn] = 
4\pi k.
$$
\end{proposition}

\begin{proof}
The lower bound follows from 
\eqref{e1.1} and \eqref{e1.2}, since
$$
E_r[\bn] \ge 4\pi r^2 Q[\bn] + 4\pi (1-r^2) |Q[\bn]| = 4\pi k.
$$
Thus it suffices to show the optimality of this bound.
The case $k=0$ is straightforward since $E_r[\boe_3]=0$.

To see the case $k=1$, 
we test the anti-skyrmion $\tilde{\bh}^\lambda(x)\colonequals\tilde{\bh}(\frac{x}{2\lambda})$ with $\la>0$, where $\tilde\bh$ is defined in \eqref{eq:defanti}. 
By computation, we have
\[ D[\tilde{\bh}^\lambda]=D[\bh^\lambda]=4\pi,\ H[\tilde{\bh}^\lambda]=0,\ V[\tilde{\bh}^\lambda]=V[\bh^\lambda]=8\pi \lambda^2.
\]
Note also that \eqref{eq:Q} implies
$
Q[\tilde{\bh}^\lambda]=Q[\tilde{\bh}]=1.
$
Therefore,
$$
\inf_{\bn\in\boM_1} E_r[\bn] 
\le \inf_{\la>0} E_r[\tilde\bh^\la] = 4\pi,
$$
which concludes the optimality when $k=1$. 

In the case $k\ge 2$, for $\la,R>0$, 
let $\tilde{\bn}^\la_R$ be the map as in \eqref{eq:defnR} with $\bh^r_R$ replaced by $\tilde\bh^\la_R$ defined in Lemma \ref{L:cptHM}. Then, the same argument in the proof of Proposition \ref{P2} implies $\tilde\bn_R^\la\in\boM$ and $Q[\tilde\bn_R^\la]=k$. Moreover, Lemma \ref{L:cptHM} yields
\begin{align*}
\inf_{\bn\in\boM_{k}} E_r[\bn] 
&\le  
\inf_{\la>0} \liminf_{R\to\infty} E_r[\tilde\bn^{\la}_R] \\
&= \inf_{\la>0} \liminf_{R\to\infty} k E_r[\tilde\bh_R^\la]\\
&= \inf_{\la>0}kE_r[\tilde\bh^\la]= 
4\pi k.
\end{align*}
Hence the proof is complete.
\end{proof}




\subsection{Unboundedness of the energy}

Finally, we address the remaining case $r>1$.
The idea is based on the proof of \cite[Theorem 2]{MR4630481} for $k=-1$, combined with our approximation lemma (Lemma \ref{L:cptHM}).

\begin{proposition}\label{P6}
When $r>1$, for any $k\in \Z$, we have
$$
\inf_{\bn\in \boM_k} E_r[\bn] =-\infty.
$$
\end{proposition}
\begin{proof}
We first consider the case when $k\le -1$. 
For $L>0$, we define
\begin{equation}\label{eq:defnL}
    \bn_L(x) \colonequals 
    \begin{cases}
        \bh_L^{1/r} (x_1,0) &\text{in } 
        \{-L\le x_1,x_2\le L\},\\
        \bh^{1/r}_L (x_1, x_2-L) & \text{in } B_L((0,L))\cap \{x_2\ge L\},\\
        \bh^{1/r}_L (x_1, x_2+L) & \text{in } B_L((0,-L))\cap \{x_2\le -L\},\\
        \bh^{1}_1 (x - b_j) & \text{in } B_1(b_j),\ 1\le j\le |k+1|,\\
        \boe_3 & \text{otherwise},
    \end{cases}
\end{equation}
where $b_j\colonequals (10(L+j),0)\in\R^2$ (see also Figure \ref{fig:1D_stretch_|k+1|_vortices}).
Then $\bn_L-\boe_3$ is continuous and compactly supported on $\R^2$, and smooth with bounded derivatives on $\R^2\setminus\{x_2=\pm L\}$.
Hence $\bn_L-\boe_3\in C_c \cap W^{1,\infty}\subset L^4\cap H^1$, and thus $\bn_L\in \boM$.

Next, we compute the degree of $\bn_L$.
Letting $q(\n)$ be the energy density of $Q$ for $\bn\in\boM$, namely
$
q(\bn)= \frac 1{4\pi} \bn\cdot \rd_1 \bn \times \rd_2 \bn,
$
we have
\begin{align*}
Q(\bn_L) &= 
\int_{B_L((0,L))\cap \{x_2\ge L\}}
q(\bh^{1/r}_{L}(x_1,x_2-L)) dx\\
&\quad +
\int_{B_L((0,-L))\cap \{x_2\le -L\}}
q(\bh^{1/r}_{L}(x_1,x_2+L)) dx\\
&\quad +
\sum_{j=1}^{|k+1|} 
\int_{B_1(b_j)}
q(\bh^{1}_{1}(x-b_j)) dx\\
&=
\int_{B_L(0)} q(\bh^{1/r}_L) dx + (-k-1) 
\int_{B_1(0)} q(\bh^{1}_1) dx = -1+ (k+1) =k.
\end{align*}
Hence $\bn_L\in \boM_{k}$.

Finally we show that the energy of $\bn_L$ diverges to $-\infty$ as $L\to\infty$.
Letting $e_r(\bn)$ be the energy density of $E_r$ for $\bn\in\boM$, we have
\begin{align*}
E_r[\bn_L]  
&=
\int_{-L}^L \int_{-L}^L 
e_r(\bh^{1/r}_L(x_1,0)) dx_1dx_2 
\\
& \quad + \int_{B_L((0,L))\cap \{x_2\ge L\}}
e_r(\bh^{1/r}_{L}(x_1,x_2-L)) dx\\
& \quad +
\int_{B_L((0,-L))\cap \{x_2\le -L\}}
e_r(\bh^{1/r}_{L}(x_1,x_2+L)) dx\\
& \quad +
\sum_{j=1}^{|k+1|} 
\int_{B_1(b_j)}
e_r(\bh^{1}_{1}(x-b_j)) dx\\
&=
L\int_{-L}^L e_r(\bh^{1/r}_L (x_1,0)) dx_1 
+ 
\int_{B_L(0)} e_r(\bh^{1/r}_L)dx + O_{L\to\infty} (1).
\end{align*}
Here, Lemma \ref{L:cptHM} implies that
$$
\int_{B_L(0)} e_r(\bh^{1/r}_L)dx = O_{L\to\infty} (1).
$$
Now we claim that
\begin{equation}\label{eq:3.6}
\int_{-L}^L e_r(\bh^{1/r}_L(x_1,0)) dx_1 \xrightarrow{L\to\infty} 
\int_\R e_r(\bh^{1/r}(x_1,0)) dx_1 
= \int_\R \frac{2(1-r^2)}{(r^2x_1^2+1)^2} dx_1 <0.
\end{equation}
First, noting that $\supp \bh^{1/r}_L(\cdot,0)\subset [-L,L]$, 
we have
\begin{align*}
&\int_{-L}^L e_r(\bh^{1/r}_L(x_1,0)) dx_1\\
&=
\int_{\R} \frac 12 |\rd_{1}\bh^{1/r}_L(x_1,0)|^2 
+
r (\bh^{1/r}_L(x_1,0)-\boe_3)\cdot 
\begin{pmatrix}
\rd_1 \\ 0 \\ 0
\end{pmatrix}
\times 
\bh^{1/r}_L(x_1,0) \\
&\qquad + \frac18 |\bh^{1/r}_L(x_1,0)-\boe_3|^4 dx_1.
\end{align*}
By Lemma \ref{L:cptHM}, we have
\begin{align*}
&\int_{\R} \frac 12 |\rd_{1}\bh^{1/r}_L(x_1,0)|^2 +
\frac18 |\bh^{1/r}_L(x_1,0)-\boe_3|^4 dx_1 \\
&\xrightarrow{L\to\infty} 
\int_{\R} \frac 12 |\rd_{1}\bh^{1/r}(x_1,0)|^2 +
\frac18 |\bh^{1/r}(x_1,0)-\boe_3|^4 dx_1.
\end{align*}
On the other hand, integration by parts yields
\begin{align*}
&\int_{\R} (\bh^{1/r}_L(x_1,0)-\boe_3)\cdot 
\begin{pmatrix}
\rd_1 \\ 0 \\ 0
\end{pmatrix}
\times 
\bh^{1/r}_L(x_1,0)\\
&=
\int_{\R} 
(h^{1/r}_{L})_2 (x_1,0) 
(-\rd_1 (h^{1/r}_{L})_3(x_1,0)) 
+
((h^{1/r}_{L})_3(x_1,0) -1)
\rd_1 (h^{1/r}_{L})_2 
(x_1,0) dx_1
\\
&= 2
\int_{\R} \rd_1 (h^{1/r}_{L})_2
(x_1,0)
((h_{L}^{1/r})_3(x_1,0)-1) dx_1 \\
&= -
 \int_{\R} \rd_1 (h^{1/r}_{L})_2 (x_1,0) |\bh^{1/r}_L(x_1,0)-\boe_3|^2 dx_1, 
\end{align*}
where we used \eqref{eq:1124_02} in the last identity. 
The same formula holds for $\bh^{1/r}$ as well. 
Therefore, Lemma \ref{L:cptHM} (iii), together with the H\"older inequality, implies
\begin{align*}
&\int_{\R} \rd_1 (h^{1/r}_{L})_2 (x_1,0) |\bh^{1/r}_L(x_1,0)-\boe_3|^2 dx_1 \\
&\xrightarrow{L\to\infty}
\int_{\R} \rd_1 h^{1/r}_2 (x_1,0) |\bh^{1/r}(x_1,0)-\boe_3|^2 dx_1,
\end{align*}
which concludes \eqref{eq:3.6}. 
Consequently, there exists $C>0$ such that 
for large enough $L$, we have
$$
\int_{-L}^L e_r(\bh^{1/r}_L(x_1,0)) dx_1 \le -C.
$$
Therefore,
$$
E_r[\bn_L] \le -CL + O_{L\to\infty} (1) \xrightarrow{L\to\infty} -\infty.
$$

When $k\ge 0$, we replace $\bh^1_1$ in the definition of $\bn_L$ by $\tilde\bh^{1}_1$. Then $\bn_L\in\boM_{k}$, and $\lim_{L\to\infty} E_r[\bn_L]=-\infty$ follows in the same manner.
\end{proof}

\begin{figure}[htbp]
\begin{center}\scriptsize
    \begin{tikzpicture}[scale=1.5]
    \fill [pattern=north west lines]
    (1,1)  arc (0:180:1) --cycle;
    \draw 
    (1,1)  arc (0:180:1) --cycle;
    \fill [pattern=north west lines]
    (-1,-1)  arc (180:360:1) --cycle;
    \draw 
    (-1,-1)  arc (180:360:1) --cycle;
    \filldraw [fill=lightgray]
	 (-1, 1)
	-- (1, 1)
	-- (1,-1)
    -- (-1,-1)
        -- cycle;
    \draw [thick] (-1,1)--(1.5,1) node [anchor=west] {\footnotesize $x_2=L$};
    \draw [thick] (-1,-1)--(1.5,-1) node [anchor=west] {\footnotesize $x_2=-L$};
    \node[fill=white] at (0,0) {$\bh^{1/r}_L(x_1,0)$};
    \node[fill=white] at (0,1.4) {$\bh^{1/r}_L(x_1,x_2-L)$};
    \node[fill=white] at (0,-1.4) {$\bh^{1/r}_L(x_1,x_2+L)$};
    \draw[-latex,very thick] (3,0.49) arc
    [
        start angle=80,
        end angle=400,
        x radius=0.5,
        y radius =0.5
    ] ;
    \draw[-latex,very thick] (5.5,0.49) arc
    [
        start angle=80,
        end angle=400,
        x radius=0.5,
        y radius =0.5
    ] ;
    \node at (4.2,0) {{\Large $\cdots$}};
    \node at (4.2,0.9) {{\normalsize $|k+1|$ vortices}};
\end{tikzpicture}
\end{center}
\caption{The construction of $\bn_L$ in \eqref{eq:defnL}.}
\label{fig:1D_stretch_|k+1|_vortices}
\end{figure}
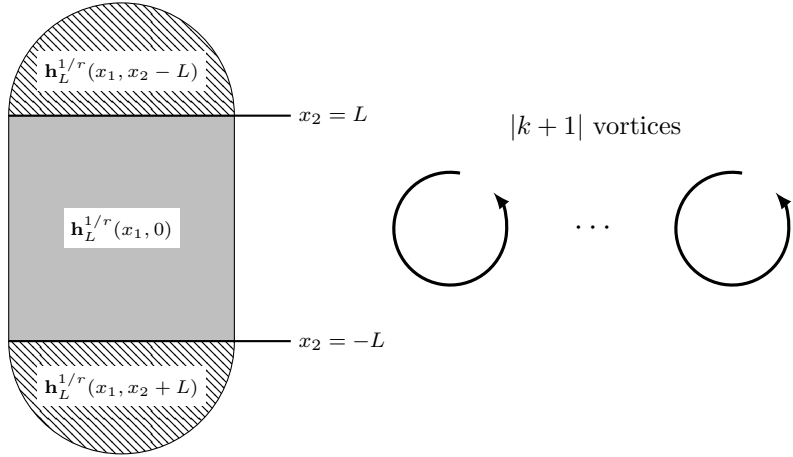

\begin{proof}[Proof of Theorem \ref{thm:main1_energy_formula}]
    It follows by Propositions \ref{P2}, \ref{P4}, and \ref{P6}.
\end{proof}

\section{Rigidity of minimizers}\label{sec:rigidity}

In this section we prove the rigidity of minimizers in Theorem \ref{thm:main4_rigidity}.
We begin with a basic characterization in the non-endpoint case.

\begin{lemma}\label{lem:minimizer_necessary}
    Let $0<r<1$, $k\in\Z$, and $\bn\in\boM_k$.
    Then $\bn$ is a minimizer of $E_r$ in $\boM_k$ if and only if $\bn$ satisfies
    \begin{align*}
    |\boD^r_1\bn + \bn\times \boD^r_2\bn| = 0 \qquad \text{and} \qquad
    D[\bn] = 4\pi|Q[\bn]|.
    \end{align*}
\end{lemma}

\begin{proof}
    Using \eqref{e1.1} and \eqref{e1.2} we have
    \begin{align*}
    E_r[\bn] &= \frac {r^2}2 \int_{\R^2} |D^r_1\bn + \bn\times D^r_2\bn|^2 dx 
    + 4\pi r^2k + (1-r^2)D[\bn]\\
    &\ge 4\pi r^2k + (1-r^2)D[\bn]\\
    &\ge 4\pi r^2k + 4\pi(1-r^2)|Q[\bn]|,
    \end{align*}
    where the last term agrees with the minimal energy computed in Theorem \ref{thm:main1_energy_formula}.
    Hence $\bn$ is a minimizer if and only if the two 
    inequalities in the middle become equality, yielding the desired identities.
\end{proof}

The equation $|\boD^r_1\bn + \bn\times \boD^r_2\bn| = 0$ is often called the \emph{Bogomol'nyi equation}.

\subsection{Nonnegative degree case}

We first address the case of nonnegative degrees.
In this case we give a fairly simple argument based solely on the following energy identity.

\begin{lemma}\label{L9}
Let $\bn\in \boM$ satisfy $|\boD^r_1 \bn + \bn\times \boD^r_2 \bn| =0$. Then
$$
D[\bn] - 4\pi Q[\bn] = \frac 1{r^2} V[\bn].
$$
\end{lemma}
\begin{proof}
By definition of the helical derivative, we have
$$
\rd_1 \bn + \bn \times \rd_2 \bn = \frac 1r \left( \boe_1 \times \bn + \bn\times (\boe_2\times \bn)
\right).
$$
Computation using $|\bn|=1$ yields
\begin{align*}
|\rd_1 \bn + \bn\times \rd_2\bn|^2 = |\nab \bn|^2 - 2\bn\cdot \rd_1\bn \times \rd_2\bn,
\end{align*}
while
$$
\boe_1\times \bn +\bn\times (\boe_2\times \bn) = 
\begin{pmatrix}
-n_1n_2 \\ 
1-n_3 - n_2^2\\
n_2(1-n_3)
\end{pmatrix}
,
$$
which implies
\begin{align*}
|\boe_1\times \bn +\bn\times (\boe_2\times \bn)|^2
&=
n_1^2n_2^2 + (1-n_3-n_2^2)^2 + n_2^2(1-n_3)^2\\
&=
(1-n_3)^2 + n_2^2\left(|\bn|^2-1\right)
= \frac 14|\bn-\boe_3|^4.
\end{align*}
Comparing and integrating these norms completes the proof.
\end{proof}

\begin{proof}[Proof of Theorem \ref{thm:main4_rigidity} for nonnegative degrees]
Suppose that $0<r<1$ and $\bn$ is a minimizer of $E_r$ in $\boM_k$ for $k\geq0$.
Then Lemma \ref{lem:minimizer_necessary} together with Lemma \ref{L9} and $Q[\bn]=k\geq0$ yields
$$
\frac 1{r^2}V[\bn] = D[\bn] -4\pi Q[\bn] =0,
$$
which implies $\bn\equiv\boe_3$, and hence $k=0$. 
\end{proof}

\subsection{Negative degree case}

Now we turn to the negative degree case.
Here we use much more sophisticated tools, namely the regularity theory for harmonic maps, and also a complex analytic approach inspired by Barton-Singer--Ross--Schroers' classification theory for the Bogomol'nyi equation \cite{MR4091507}.

\begin{proof}[Proof of Theorem \ref{thm:main4_rigidity} for negative degrees]
    Suppose that $0<r<1$ and $\bn$ is a minimizer of $E_r$ in $\boM_k$ with $k<0$.
    Then, 
    the second identity in Lemma \ref{lem:minimizer_necessary} together with \eqref{eq:factD} and $Q[\bn]=k< 0$, as well as with the known regularity theory for harmonic maps (see, e.g., \cite[Lemma A.2]{MR728866}), implies that
    the map $\bn$ is an anti-holomorphic harmonic map from $\R^2 \simeq \C$, where $z=x_1+ix_2$, to the Riemann sphere $\S^2\simeq\C\cup\{\infty\}$ (regarding the northpole as $\infty$).
    
    Set $X_\pm = \{x\in\R^2 \mid \bn(x)=\pm \boe_3\}$. 
    Then $U\colonequals \R^2 \setminus \bigcup_{\pm} X_\pm$ is a non-empty connected open subset of $\R^2$; indeed, $X_\pm$ consists only of at most isolated points since $\bn$ is nonconstant in view of $Q[\bn]\neq 0$. 
    
    

    Following \cite{MR4091507}, we define $v:U\to\C$ by
    \begin{align*}
        v\colonequals \frac{1+n_3}{n_1+in_2},
    \end{align*}
    a stereographic projection from the northpole.
    Then the Bogomol'nyi equation $|\boD^r_1\bn + \bn\times \boD^r_2\bn| = 0$ is transformed into
    $\rd_{\ovl{z}} v = -i/2r$, from which we can write
    \begin{equation}\label{familysol}
    v= -\frac{i}{2r} \ovl{z} + f(z)
    \end{equation}
    for some holomorphic function $f$ on $U$. On the other hand, since $\bn$ is anti-holomorphic, it follows that $f$ is a constant map.
    Noting that
    \begin{equation}\label{eq:v_to_n}
        n_1+in_2=\frac{2\ovl{v}}{|v|^2+1}, \quad n_3=\frac{|v|^2-1}{|v|^2+1}
    \end{equation}
    we can write, for some $x_0\in \R^2$,
    $$
    \bn (x) = \bh^{r}(x-x_0)\quad \text{on } U.
    $$
    Since $\bn$ is anti-holomorphic, this identity extends to the whole $\R^2$, and thus in particular $X_+ = \emptyset$ and $X_-=\{x_0\}$.
    If $k=-1$, this implies the uniqueness of the minimizer of $\inf_{\boM_{-1}}E_r$, while a contradiction follows when $k\le -2$.
\end{proof}

\begin{remark}
    The above approach also works for nonnegative degrees, but our proof based on Lemma \ref{L9} in the previous section is completely self-contained and much simpler.
\end{remark}



\section{Stability of the homogeneous state}\label{sec:stability}

In this section, we analyze the stability of the homogeneous state $\bn = \boe_3$ for the energy functional $E_{r,h}$ in \eqref{eq:E_rh}.
By \eqref{eq:1124_02} we have
\begin{equation}\label{eq:ZL^2}
	Z[\bn] = \int_{\R^2}(1-n_3) dx = \frac12 \int_{\R^2}|\bn-\boe_3|^2 dx.
\end{equation}  
Recall that, when $h \neq 0$, the natural space for $E_{r,h}$ is $\boM'$ in \eqref{eq:space_M'}, which is a proper subset of $\boM$. 

 
We proceed to the stability analysis: first the critical $h=0$, then supercritical $h>0$, and finally subcritical case $h<0$.

\subsection{The critical case} 

In order to address the stability in the critical case $h=0$, we adopt a direct variational approach. We estimate the total energy of configurations that are close to $\boe_3$ with respect to the $L^2$-norm, and show that any $L^2$-perturbation leads to a strict increase in energy.

\begin{proof}[Proof of Theorem~\ref{thm:main2_stability}]
Let us write $\bn = \boe_3 + \bphi$. A straightforward computation yields
\begin{equation}\label{energydifference*}
E_r[\bn] - E_r[\boe_3] = 
\frac 12 \int_{\R^2} |\nab \bphi|^2 dx 
+ r\int_{\R^2} \bphi \cdot \nab \times \bphi dx 
+ \frac 18 \int_{\R^2} |\bphi|^4 dx.
\end{equation}
Note that
\begin{equation}\label{eq:phi_3}
1= |\bn|^2 = \phi_1^2 + \phi_2^2 + (1+\phi_3)^2 \quad
\Longleftrightarrow
\quad
\phi_3 = -\frac 12 |\bphi|^2.
\end{equation}
Then
\begin{align}
r \left| \int_{\R^2} \bphi \cdot \nab \times \bphi dx\right| 
&\stackrel{\eqref{eq:H_int_by_parts}}{=} 2
r \left|\int_{\R^2} \phi_3 (\rd_1 \phi_2 - \rd_2 \phi_1) dx\right| \notag\\
&\le 2
r \nor{\phi_3}{L^2} \nor{\rd_1 \phi_2 - \rd_2 \phi_1}{L^2}\notag \\
&\stackrel{\eqref{eq:phi_3}}{\le}
2r \nor{\frac 12 |\bphi|^2}{L^2} \nor{\nab \bphi}{L^2} \notag \\
&=
r \nor{\bphi}{L^4}^2 \nor{\nab \bphi}{L^2}\label{ineq:phi}\\
&\stackrel{\eqref{eq:Ladyzhenskaya}}{\le} Cr \nor{\bphi}{L^2} \nor{\nab \bphi}{L^2}^2.\notag
\end{align}
Inserting this estimate into \eqref{energydifference*}, we obtain
\begin{align*}
E_r [\bn] - E_r[\boe_3] \geq \left(\frac{1}{2}-Cr\nor{\bphi}{L^2}\right)\int_{\R^2} |\nab \bphi|^2 dx +\frac18\int_{\R^2} |\bphi|^4 dx.
\end{align*}
Now, if we take $\delta=\frac1{4Cr}$, the assumption $\nor{\bn-\boe_3}{L^2}=\nor{\bphi}{L^2}\leq\delta$ yields
\begin{align*}
E_r [\bn] - E_r[\boe_3] \geq \frac{1}{4}\nor{\nabla(\bn-\boe_3)}{L^2}^2 + \frac{1}{8}\nor{\bn-\boe_3}{L^4}^4,
\end{align*}
which implies the desired strict local minimality of $\bn=\boe_3$.
\end{proof}


\subsection{The supercritical case}

We now address the stability in the supercritical case $h>0$. We follow a similar strategy as in the critical case. We estimate the total energy of configurations that are close to $\boe_3$ with respect to the $L^4$-norm (instead of the $L^2$-norm), thus showing that any admissible perturbation leads to a strict increase in energy.

\begin{theorem}\label{thm:h>0}
Let $r>0$ and $h>0$. Then there exists $\delta=\delta(r,h)>0$ such that
$$
E_{r,h}[\bn] > E_{r,h}[\boe_3]
$$
holds for all $\bn\in \boM'$ such that $0<\|\bn-\boe_3\|_{L^4(\R^2)}\leq \delta$.
\end{theorem}
\begin{proof}
Let us write $\bn = \boe_3 + \bphi$. Recalling \eqref{eq:ZL^2}, a straightforward computation yields
\begin{multline}\label{energydifference2}
E_{r,h} [\bn] - E_{r,h}[\boe_3] = 
\frac 12 \int_{\R^2} |\nab \bphi|^2 dx 
\\+ r\int_{\R^2} \bphi \cdot \nab \times \bphi dx 
+ \frac 18 \int_{\R^2} |\bphi|^4 dx 
+ \frac{h}2  \int_{\R^2} |\bphi|^2 dx.
\end{multline}
Arguing exactly as in the proof of Theorem~\ref{thm:main2_stability}, we find that \eqref{ineq:phi} holds. By Sobolev embedding, we have $\nor{\bphi}{L^4}\leq C\nor{\bphi}{H^1}$ and hence
$$
r \left| \int_{\R^2} \bphi \cdot \nab \times \bphi dx\right| 
\stackrel{\eqref{ineq:phi}}{\le} r \nor{\bphi}{L^4}^2 \nor{\nab \bphi}{L^2}\le  C r \nor{\bphi}{L^4}\nor{\bphi}{H^1}^2.
$$
Inserting the above estimate into \eqref{energydifference2}, we obtain 
$$
E_{r,h} [\bn] - E_{r,h}[\boe_3] \ge 
\left(\frac12\min\{1,h\} - Cr\nor{\bphi}{L^4} \right)
\nor{\bphi}{H^1}^2
+ \frac 18 \int_{\R^2} |\bphi|^4 dx.
$$
Now, if we take $\delta=\frac{\min\{1,h\}}{4Cr}$, the assumption $\nor{\bn-\boe_3}{L^4}=\nor{\bphi}{L^4}\leq\delta$ yields
\begin{align*}
	E_{r,h} [\bn] - E_{r,h}[\boe_3] \geq \frac{\min\{1,h\}}{4}\nor{\bn-\boe_3}{H^1}^2 + \frac{1}{8}\nor{\bn-\boe_3}{L^4}^4,
\end{align*}
which implies the desired strict local minimality of $\bn=\boe_3$.
\end{proof}

A direct consequence of Theorem~\ref{thm:h>0} and \eqref{eq:Ladyzhenskaya} is the following

\begin{corollary}\label{corollary:stabilityh>0}
Let $r>0$ and $h>0$. Then $\boe_3$ is a strict local minimizer of $E_{r,h}$ in $\boM'$.
\end{corollary}

\subsection{The subcritical case}

We finally address the instability in the subcritical case $h<0$. Here we construct a one‑parameter family of smooth, compactly supported perturbations of $\boe_3$ whose energy is strictly lower.

\begin{theorem}\label{thm:h<0}
Let $r>0$ and $h<0$. Then there exists $\bphi\in C_c^\infty(\R^2:\R^3)$ such that
$$
E_{r,h}\left[\frac{\boe_3+t\bphi}{|\boe_3+t\bphi|}\right] < E_{r,h}[\boe_3] = 0
$$
for all sufficiently small $t\in\R$ with $t\neq0$.
\end{theorem}

\begin{proof}
We first claim that there exists $\bphi\in C_c^\infty(\R^2:\R^3)$ such that
\begin{equation}\label{eq:condphi}
\bphi(x)\cdot \boe_3=0\quad (\forall x\in\R^2) \quad\mathrm{and}\quad \int_{\R^2} |\nab \bphi|^2 + h |\bphi|^2 dx <0.
\end{equation}
To see this, let $\bphi_0\in C_c^\infty(\R^2:\R^3)$ be a nonzero vector field with $\bphi_0\cdot \boe_3=0$. For $\la>0$, define
$$
\bphi_\la(x)= \bphi_0 \left(\frac{x}{\la}\right) \quad (\forall x\in\R^2).
$$
Then $\bphi_\la\cdot \boe_3=0$ for every $\la>0$, and by a change of variables we obtain
$$
\int_{\R^2} |\nab \bphi_\la|^2 + h |\bphi_\la|^2 dx = \int_{\R^2} |\nab \bphi|^2 + h\la^2 |\bphi|^2 dx \xrightarrow{\la\to\infty} -\infty.
$$
Thus, by choosing $\bphi=\bphi_\la$ with $\la$ sufficiently large (and fixed), we obtain a vector field that satisfies \eqref{eq:condphi}.

Now fix such a vector field $\bphi$ and, for small $t\in\R$, define
$$
\bn_t = \frac{\boe_3+t\bphi}{|\boe_3+t\bphi|}.
$$
We note that $\bn_t\in \boM'$, since $|\bn_t|=1$ and $\bn_t-\boe_3$ is smooth with compact support.

The first condition in \eqref{eq:condphi} implies that, for sufficiently small $t$, we have
$$
\frac{1}{|\boe_3 + t\bphi|}=1-\frac{t^2}{2}|\bphi|^2+O(t^3),
$$
which in turn yields
$$
\bn_t = \boe_3+ t \bphi -\frac{t^2}{2}|\bphi|^2 \boe_3+ O(t^3).
$$
A straightforward computation then shows that, for sufficiently small $t$,
$$
E_{r,h}[\bn_t]=\frac{t^2}{2}\int_{\R^2}\left(|\nabla \bphi|^2 +h|\bphi|^2\right)dx+O(t^3).
$$
Therefore, in view of the second condition in \eqref{eq:condphi}, we conclude that
$$
E_{r,h}[\bn_t]<E_{r,h}[\boe_3]=0,
$$
for all sufficiently small $t$ with $t\neq 0$. The proof is thus complete.
\end{proof}

The following result is a direct consequence of Theorem~\ref{thm:h<0} and the fact that $dE_{r,h}[\boe_3] = 0$.

\begin{corollary}\label{corollary:stabilityh<0}
Let $r>0$ and $h<0$. Then $\boe_3$ is an unstable critical point of $E_{r,h}$ in $\boM'$.
\end{corollary}

\begin{proof}[Proof of Theorem \ref{thm:main3_staiblity_transition}]
    It follows by Corollaries \ref{corollary:stabilityh>0} and \ref{corollary:stabilityh<0}.
\end{proof}


\appendix

\section{Counterexamples in the endpoint case}\label{sec:couterexample_endpoint}

Here we rigorously verify that substantially new minimizers emerge in the endpoint case $r=1$, in stark contrast to the case $0<r<1$ treated in Section \ref{sec:rigidity}.
In particular, they give counterexamples to the uniqueness for degrees $Q=0,-1$ and also the nonexistence for $Q\geq1$ in Theorem \ref{thm:main4_rigidity}.
We stress that the non-uniqueness in this context means that two maps in $\boM$ do not agree even up to the natural invariances; translations and rotations in the domain, and horizontal rotations in the codomain.
Most of the arguments are strongly inspired by the complex analytic approach developed in \cite{MR4091507}.

We first observe that the same computation as in Lemma \ref{lem:minimizer_necessary} yields the following characterization, the proof of which can be safely omitted.

\begin{lemma}\label{lem:minimality_Bog_r=1}
    Let $k\in\Z$ and $\bn\in\boM_k$.
    Then $\bn$ is a minimizer of $E_1$ in $\boM_k$ if and only if 
    \begin{equation}\label{eq:Bogomol'nyi}
    |\boD^1_1 \bn + \bn\times \boD^1_2 \bn| =0.
    \end{equation}
\end{lemma}

In addition, if we let $v= \frac{1+n_3}{n_1+in_2}$ with variable $z=x_1+ix_2$ as in Section \ref{sec:rigidity}, then a family of a.e.\ solutions to the Bogomol'nyi equation \eqref{eq:Bogomol'nyi} is given by \eqref{familysol} with meromorphic $f$, as discovered in \cite{MR4091507}.
Recall that if $v$ is smooth up to isolated singularities, then so is the resulting map $\bn$, see \eqref{eq:v_to_n}.
Therefore, inserting any nontrivial meromorphic function $f$ into \eqref{familysol} yields a new minimizer, whenever $\bn$ has the integrability $D[\bn]+V[\bn]<\infty$.
Note that the integrands of $D$ and $V$ are re-expressed as
\begin{align}\label{eq:integrands_v}
\frac 12|\nab \bn|^2 = \frac{2|\nab v|^2}{(1+|v|^2)^2}, \qquad 
\frac 12 (1-n_3)^2 =  \frac{2}{(1+|v|^2)^2}.
\end{align}

In what follows, we consider a special class of meromorphic $f$, namely $f(z)=az^{k}$ with $k\in\Z$ and $a\in\C$, as well as the associated map as in \eqref{eq:v_to_n}, i.e.,
\begin{equation}\label{eq:general_formula_n_and_v_for_f=az^k}
    \bn^a=\left(\frac{2\Re(v)}{|v|^2+1},\frac{-2\Im(v)}{|v|^2+1},\frac{|v|^2-1}{|v|^2+1}\right), \qquad v \colonequals -\frac{i}{2}\overline{z} + az^{k}.
\end{equation}
Note that the power $k$ of $f$ will be strongly related to the topological degree of $\bn^a$, but generally they do not agree exactly.
One may generate other examples by considering general meromorphic functions, which is omitted in the present argument. Readers interested in this direction can refer to \cite{MR4091507}.

\subsection{The case \texorpdfstring{$k=1$}{k=1}}

As observed in \cite{MR4091507}, the integrability issue is more delicate when the meromorphic function $f$ has linear growth at infinity.
Therefore, we first address the delicate case $k=1$, i.e., $f(z)= az$ with $a=a_1+ia_2\in\C$. 
Then $v$ in \eqref{eq:general_formula_n_and_v_for_f=az^k} can be written as 
\begin{align*} 
v= X_1^a + iX_2^a,&& X^a = 
\begin{pmatrix} 
X_1^a \\
X_2^a 
\end{pmatrix} 
= 
\begin{pmatrix} 
a_1 & -\frac12-a_2 \\ 
-\frac12+a_2 & a_1 
\end{pmatrix} 
\begin{pmatrix} 
x_1 \\
x_2 
\end{pmatrix},
\end{align*} 
and the associated map $\bn^a$ can be expressed by the \emph{distorted skyrmion} (cf.\ \cite{MR4091507}) 
\begin{align}\label{eq:n^a_k=1}
\bn^a = \left( \frac{2X_1^a}{|X^a|^2+1}, \frac{-2X_2^a}{|X^a|^2+1}, \frac{|X^a|^2-1}{|X^a|^2+1} \right). 
\end{align} 
Note that the $a=0$ case corresponds to the standard skyrmion $\bh^1$, while for $a\neq0$ the map $\bn^a$ is clearly distinct from $\bh^1$ as the domain coordinates are linearly distorted. 
In particular, the anti-skyrmion $\tilde{\bh}^1$ can formally be approached by $\bn^{a}(\frac{x}{|a|})$ with $a=(a_1,0)$ as $a_1\to-\infty$. 

These maps $\bn^a$ provide new minimizers for degrees $Q=\pm1$.

\begin{proposition}
    Let $a\in\C$ and $\bn^a$ be the map defined in \eqref{eq:n^a_k=1}.
    If $|a|<\frac12$, then $\bn^a\in \boM_{-1}$, and if $|a|>\frac12$, then $\bn^a\in \boM_{1}$.
    In particular, all members of the family $\{\bn^a\}_{|a|<\frac12}$ (resp.\ $\{\bn^a\}_{|a|>\frac12}$) are minimizers of $E_1$ in $\boM_{-1}$ (resp.\ $\boM_1$).
\end{proposition}

\begin{proof}
Since 
$|\nab v|^2 = 2(|\rd_z v|^2+ |\rd_{\ovl z} v|^2)=2|a|^2 + \frac 12$, 
the integrands in \eqref{eq:integrands_v} are given by
\begin{align*}
\frac{2|\nab v|^2}{(1+|v|^2)^2} 
= \frac{4|a|^2+1}{(1+|X^a|^2)^2}, \qquad \frac{2}{(1+|v|^2)^2}
=
 \frac{2}{(1+|X^a|^2)^2}.
\end{align*}
The transformation $x\mapsto X^a$ is linear with $\det\frac{\rd X^a}{\rd x}=|a|^2-\frac{1}{4}\neq0$ (since $|a|\neq\tfrac12$).
Hence these terms are integrable in $x$ if and only if they are integrable in $X^a$, implying that $\bn^a\in\boM$.

Moreover, direct computation gives 
\begin{align*}
Q[\bn^a]
=
\int_{\R^2} \frac{1}{\pi} 
\frac{|\rd_z v|^2 - |\rd_{\ovl z}v|^2}{(1+|v|^2)^2}dx
&=
\int_{\R^2} \frac{1}{4\pi}\frac{4|a|^2-1}{(1+|X^a|^2)^2}dx.
\end{align*}
Using $dx=|\det(\rd X^a/\rd x)|^{-1}
=\frac{4}{|4|a|^2-1|}$, we obtain
\begin{align*}
Q[\bn^a]
&=
\frac{4|a|^2-1}{|4|a|^2-1|}
\int_{\R^2}\frac{1}{\pi}\frac{1}{(1+|X^a|^2)^2}dX^a
=
\begin{cases}
    -1 & (|a|<\frac12),\\
    1 & (|a|>\frac12).
\end{cases}
\end{align*}
Since the minimality of $\bn^a$ follows from Lemma \ref{lem:minimality_Bog_r=1}, the proof is complete.
\end{proof}

\subsection{The case \texorpdfstring{$k\neq1$}{k≠1}}

Now we turn to the case $k\neq1$.
In this case we have the following general result.

\begin{proposition}
    Let $f=p/q$ be a meromorphic function with coprime polynomials $p$ and $q$ such that $\deg(p)\neq\deg(q)+1$.
    Let $\bn$ be the associated map with $v=-\frac{i}{2}\ovl{z}+f(z)$ via $v=\frac{1+n_3}{n_1+in_2}$.
    \begin{itemize}
        \item If $\deg(p)>\deg(q)+1$, then $\bn\in\boM_{\deg(p)}$ and $E_1[\bn]=4\pi\deg(p)$.
        \item If $\deg(p)<\deg(q)+1$, then $\bn\in\boM_{\deg(q)-1}$ and $E_1[\bn]=4\pi(\deg(q)-1)$.
    \end{itemize}
\end{proposition}

\begin{proof}
    This follows almost directly from \cite[Lemma 4.2]{MR4091507} up to considering the correction term in the helicity $H$.
    For the reader's convenience we sketch the argument.

    To show the integrability of \eqref{eq:integrands_v}, we first note that as $f$ is meromorphic, the issue is reduced to the behavior at infinity; indeed, although $v$ may have finitely many poles, in terms of the projection from the southpole $w\colonequals \frac{n_1+in_2}{1+n_3}$, 
    the associated map $\bn$ is clearly smooth around the poles of $v$ (where $w=0$) since
    \[
    n_1+in_2=\frac{2w}{1+|w|^2}, \qquad n_3=\frac{1-|w|^2}{1+|w|^2}.
    \]
    Let $m\colonequals \deg(p)$ and $d\colonequals \deg(q)$.
    As $|z|\to\infty$, we observe that
    \[
    |v|^2 =\left|-\frac{i}2\ovl{z} +f(z)\right|^2 \sim |z|^{\max\{2,2(m-d)\}}
    \]
    thanks to $m-d\neq1$, while
    \[
    |\nabla v|^2=2(|\rd_z v|^2 + |\rd_{\ovl z}v|^2)=O(|z|^{2(m-d-1)})+\frac{1}{2}.
    \]
    Hence the integrands in \eqref{eq:integrands_v} always have the decay $O(|z|^{-4})$, yielding the desired integrability $\bn\in\boM$.

    The integrability together with Lemma \ref{lem:minimality_Bog_r=1} implies that $\bn$ is a minimizer of $E_1$ in the class $\boM_{Q[\bn]}$, i.e., $E[\bn]=4\pi Q[\bn]$.

    The remaining task is to compute the energy $E_1$ of $\bn$, which also determines the topological degree of $\bn$ through the relation $E[\bn]=4\pi Q[\bn]$.
    Now, let $E_\mathrm{BRS}$ denote the total energy in \cite{MR4091507} with $\alpha=0$ and $\kappa=1$.
    This energy agrees with our total energy $E_r$ with $r=1$ up to the correction term in the helicity $H$; more precisely, for $\bn\in\boM$,
    \begin{align}\label{eq:BRS_energy}
        E_\mathrm{BRS}[\bn] = E_1[\bn] + \int_{\R^2} \boe_3\cdot(\nabla\times\bn)dx,
    \end{align}
    whenever the last term is well defined (as an improper integral).
    In addition, in \cite[Lemma 4.2]{MR4091507} the energy $E_\mathrm{BRS}$ for the map $\bn$ under consideration is computed as
    \[
    E_\mathrm{BRS}[\bn]=4\pi\max\{m,d+1\}.
    \]
    Therefore it suffices to show that
    \begin{align*}
        \int_{\R^2} \boe_3\cdot(\nabla\times\bn)dx =
        \begin{cases}
            0 & (m>d+1),\\
            8\pi & (m<d+1).
        \end{cases}
    \end{align*}
    This follows since we compute
    \begin{align*}
        \int_{\R^2} \boe_3\cdot(\nabla\times\bn)dx = \lim_{R\to\infty}\oint_{|x|=R}(n_1dx_1+n_2dx_2) = \lim_{R\to\infty}\Re\oint_{|z|=R}\frac{2v}{1+|v|^2}dz,
    \end{align*}
    and if $m>d+1$ then the integral vanishes as $O(R^{-(m-d)}\cdot R)\to0$, whereas if $m<d+1$ then the term $-\frac{i}{2}\ovl{z}$ is dominant and hence the integrand is given by $\frac{iz}{1+R^2/4}$ up to a small error term, yielding the limit $\frac{R}{1+R^2/4}\cdot 2\pi R \to 8\pi$.
\end{proof}

The above proposition yields the following direct corollaries.

\begin{corollary}[Case $k\geq2$]
    Let $a\in\C$ with $a\neq0$ and $k\in\Z$ with $k\geq2$.
    Let $\bn^a$ be the map defined in \eqref{eq:general_formula_n_and_v_for_f=az^k}.
    Then all members of the family $\{\bn^a\}_{|a|>0}$ are minimizers of $E_1$ in $\boM_k$.
\end{corollary}

\begin{corollary}[Case $k\leq0$]
    Let $a\in\C$ with $a\neq0$ and $k\in\Z$ with $k\le 0$. 
    Let $\bn^a$ be the map defined in \eqref{eq:general_formula_n_and_v_for_f=az^k}. 
    Then all members of the family $\{\bn^a\}_{|a|>0}$ are minimizers of $E_1$ in $\boM_{-k-1}$. 
\end{corollary}

Now we gain insight into the geometric structure of the map $\bn^a=(n^a_1,n^a_2,n^a_3)$ defined in \eqref{eq:general_formula_n_and_v_for_f=az^k}, focusing on the preimages of the southpole and the equator.
More precisely, we consider
\[
Z_0=\{ n^a_3=-1 \}=\{z\in\C \mid |v|=0 \}, \quad Z_1=\{n^a_3=0\}=\{z\in\C \mid |v|=1 \}.
\]
(The northpole is not attained for $k\geq0$ and is attained only at the origin for $k\leq-1$.)
In what follows we will often use the formulae 
\begin{align}
    v(re^{i\theta}) &= -\frac{i}{2}re^{-i\theta}+ar^ke^{ik\theta}, \label{eq:v_formula} \\
    |v(re^{i\theta})|^2 &= \frac{1}{4}r^2+|a|^2r^{2k}-|a|r^{k+1}\sin\big( (k+1)\theta+\arg{a} \big). \label{eq:|v|^2_formula}
\end{align}

Since the case $k=0$ corresponds to translated skyrmions (as already discussed in Section \ref{sec:rigidity}), we only discuss $k\geq2$ and $k\leq-1$.
We first discuss the special case $k=-1$ and then address the general case $|k|\geq2$, assuming $a\neq0$ throughout.

\subsubsection{The case $k=-1$}
In this special case, the geometry of the sets $Z_0$ and $Z_1$ is rather simple.
Let $a=a_1+ia_2\neq0$.
By \eqref{eq:v_formula} we have $z\in Z_0$ if and only if $2a=i|z|^2$.
Hence
\begin{align*}
    Z_0=
    \begin{cases}
       \{ |z|^2=2a_2 \} & (a_1=0,\ a_2>0),\\
       \emptyset & (\text{otherwise}).
    \end{cases}
\end{align*}
By \eqref{eq:|v|^2_formula} we have $z\in Z_1$ if and only if $\frac{1}{4}|z|^4-(\Im{a}+1)|z|^2+|a|^2=0$.
Solving the equation yields
\begin{align*}
    Z_1=
    \begin{cases}
       \big\{ |z|^2=2 ( a_2+1 \pm \sqrt{2a_2+1-a_1^2} ) \big\} & (a_2>\frac{1}{2}a_1^2-\frac{1}{2}),\\
       \{ |z|^2=2(a_2+1) \} & (a_2=\frac{1}{2}a_1^2-\frac{1}{2}),\\
       \emptyset & (a_2<\frac{1}{2}a_1^2-\frac{1}{2}).
    \end{cases}
\end{align*}
Thus the set $Z_0\cup Z_1$ consists of at most three concentric circles.
In particular, if the circle $Z_0$ is present, then $Z_1$ always has two circles surrounding $Z_0$.

\subsubsection{The case $k\geq2$}
By \eqref{eq:v_formula} the set $Z_0$ is computed as
\[
Z_0=\{0\}\cup\{\gamma e^{i\frac{2\pi j}{k+1}} \mid j=0,1,\dots,k\},
\]
where $\gamma\colonequals (2|a|)^{-\frac{1}{k-1}}e^{i\frac{\pi/2-\arg{a}}{k+1}}\in\C$.
The nonzero part of $Z_0$ consists of $k+1$ points uniformly distributed on the circle $\{|z|=(2|a|)^{-\frac{1}{k-1}}\}$.
On the other hand, by \eqref{eq:|v|^2_formula} the set $Z_1$ turns out to be a nontrivial planar curve with $(k+1)$-fold rotational symmetry given by
\begin{align}\label{eq:Z_1}
    Z_1=\left\{ re^{i\theta} \in \C^\times \,\left|\, |a|r^{k-1}+\frac{r^{-k+1}}{4|a|}-\frac{r^{-k-1}}{|a|} = \sin\big( (k+1)\theta + \arg{a} \big) \right\}\right.,
\end{align}
The set $Z_0\cup Z_1$ exhibits qualitatively distinct shapes depending on the parameter $a$, see Figure~\ref{k=2} for $k=2$ and Figure~\ref{k=5} for $k=5$.

\begin{figure}[htbp]
	\centering
	\begin{subfigure}{0.32\textwidth}
		\includegraphics[width=\linewidth]{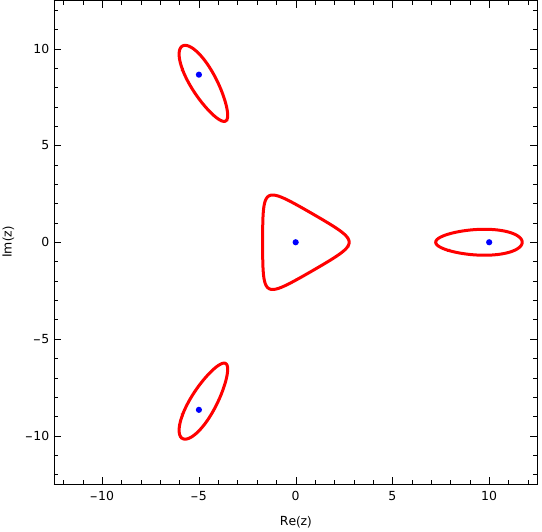}
		\caption{$a=i/20$}
	\end{subfigure}
	\begin{subfigure}{0.32\textwidth}
		\includegraphics[width=\linewidth]{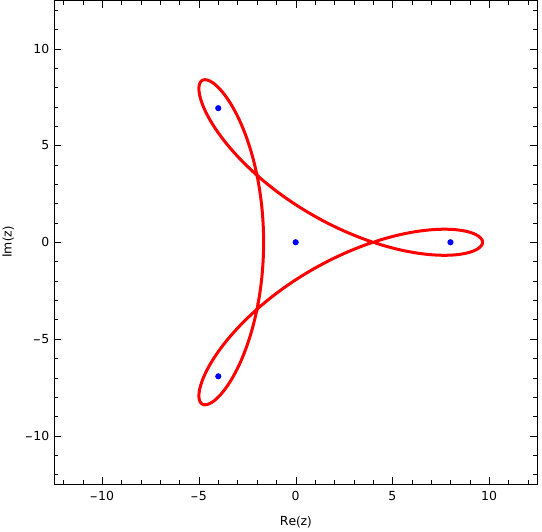}
		\caption{$a=i/16$ ($=ia_*$)}
	\end{subfigure}
	\begin{subfigure}{0.32\textwidth}
		\includegraphics[width=\linewidth]{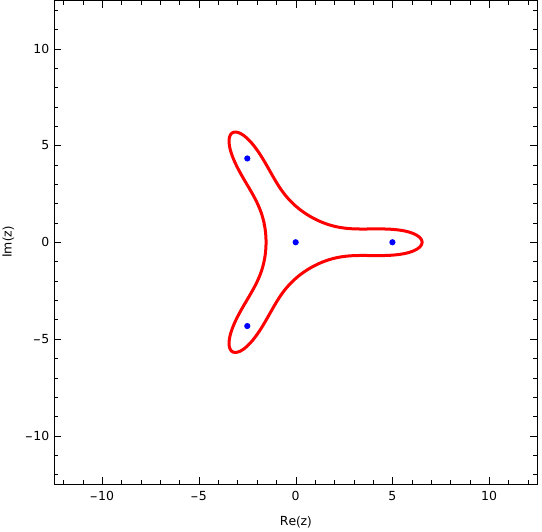}
		\caption{$a=i/10$}
	\end{subfigure}
	\caption{Plots of $Z_0$ (blue) and $Z_1$ (red) for $v=-\frac{i}{2}\overline{z} + az^{2}$ with different values of $a$.}
    \label{k=2}
\end{figure}

\begin{figure}[htbp]
	\centering
	\begin{subfigure}{0.32\textwidth}
		\includegraphics[width=\linewidth]{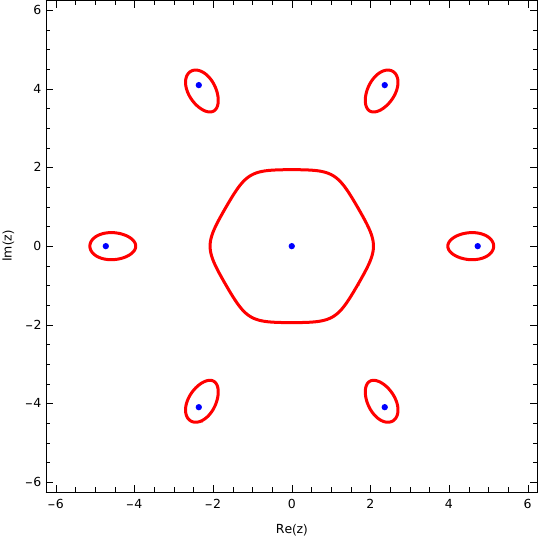}
		\caption{$a=i/1000$}
	\end{subfigure}
	\begin{subfigure}{0.32\textwidth}
		\includegraphics[width=\linewidth]{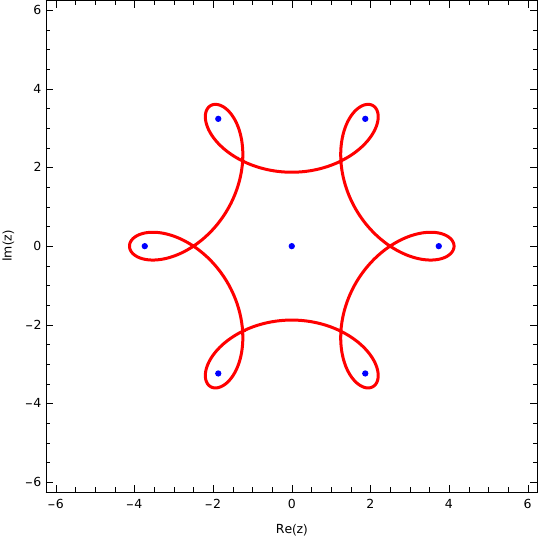}
		\caption{$a=8i/3125$ ($=ia_*$)}
	\end{subfigure}
	\begin{subfigure}{0.32\textwidth}
		\includegraphics[width=\linewidth]{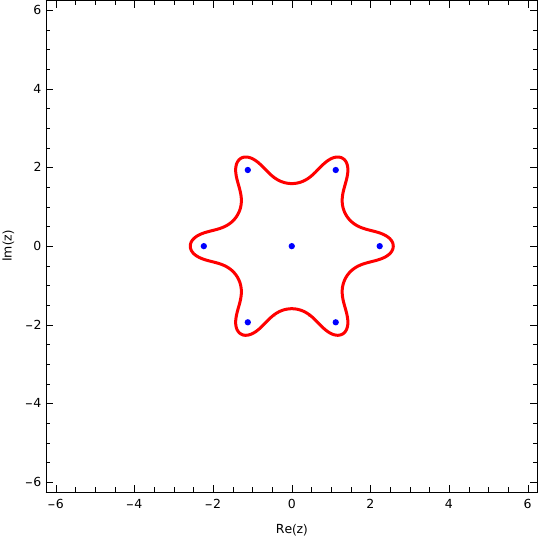}
		\caption{$a=i/50$}
	\end{subfigure}
	\caption{Plots of $Z_0$ (blue) and $Z_1$ (red) for $v=-\frac{i}{2}\overline{z} + az^{5}$ with different values of $a$.}
    \label{k=5}
\end{figure}

In particular, the figures suggest that for each $k\geq2$ there is a critical value $a_*>0$.
If $|a|=a_*$ then the set $Z_1$ is a single immersed closed curve.
For $|a|> a_*$ the curve $Z_1$ transforms into an embedded closed curve, whereas for $|a|<a_*$ it bifurcates into $k+2$ embedded closed curves, each of which encloses one point in $Z_0$.
The threshold is explicitly given by
\begin{equation}\label{eq:threshold_a_*}
    a_*=\frac{(k-1)^{k-1}}{(2k)^k},
\end{equation}
which can be computed by considering when $\partial_r|v|^2=\partial_\theta|v|^2=0$ holds at some point in $Z_1$.
As $|a|\to0$, the unique central closed curve converges to the circle $\{|z|=2\}$ and all the others escape to infinity, which agrees with the formal observation that $\bn^a$ converges to the standard skyrmion $\bh^1$.
As $|a|\to\infty$, the sets $Z_0$ and $Z_1$ concentrate on the origin, and elsewhere $\bn^a$ converges to the homogeneous state $\boe_3$.

\subsubsection{The case $k\leq-2$}
By \eqref{eq:v_formula} the set $Z_0$ is similarly computed as
\[
Z_0=\{\gamma e^{i\frac{2\pi j}{k+1}} \mid j=0,1,\dots,|k|-2\},
\]
with the same radius $\gamma=(2|a|)^{-\frac{1}{k-1}}e^{i\frac{\pi/2-\arg{a}}{k+1}}\in\C$ as before.
Here the origin corresponds to the northpole, thus not included in $Z_0$ as opposed to the case $k\geq2$.
The set $Z_1$ is given in the exactly same form \eqref{eq:Z_1}, which possesses $(|k|-1)$-fold rotational symmetry.
See Figure~\ref{k=-2} for $k=-2$, Figure~\ref{k=-3} for $k=-3$, Figure~\ref{k=-4} for $k=-4$, and Figure~\ref{k=-5} for $k=-5$.
In this case the threshold value $a_*$ can be computed as
\[
a_*=\frac{1}{2|k|}\left(\frac{k-1}{2k}\right)^{k-1}.
\]
Here, if $|a|=a_*$ then $Z_1$ consists of either a single immersed closed curve (for $k$ even) or two intersecting closed curves (for $k$ odd).
For $|a|<a_*$ the set $Z_1$ branches into two nested curves, while for $|a|>a_*$ it bifurcates into $|k|-1$ non-nested closed curves.
Notice that the above expression of $a_*$ agrees with \eqref{eq:threshold_a_*} if $k\geq2$.

\begin{figure}[htbp]
	\centering
	\begin{subfigure}{0.32\textwidth}
		\includegraphics[width=\linewidth]{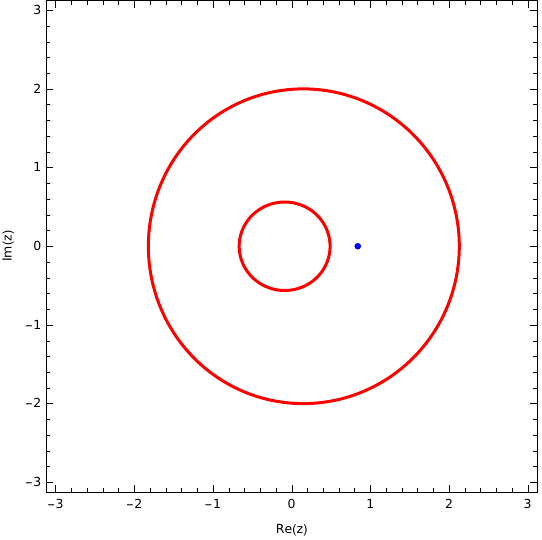}
		\caption{$a=8i/27$}
	\end{subfigure}
	\begin{subfigure}{0.32\textwidth}
		\includegraphics[width=\linewidth]{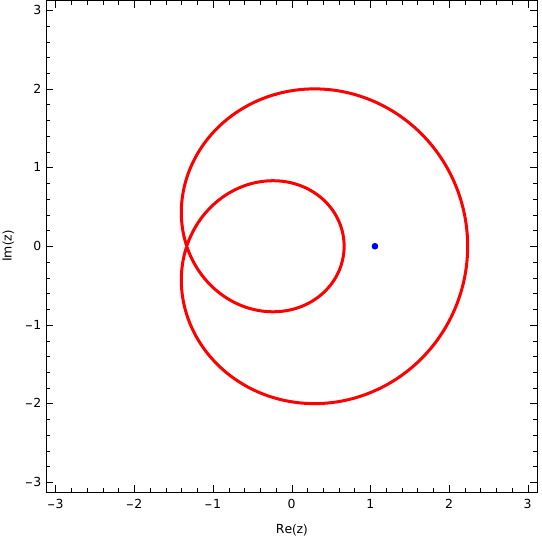}
		\caption{$a=16i/27$ ($=ia_*$)}
	\end{subfigure}
	\begin{subfigure}{0.32\textwidth}
		\includegraphics[width=\linewidth]{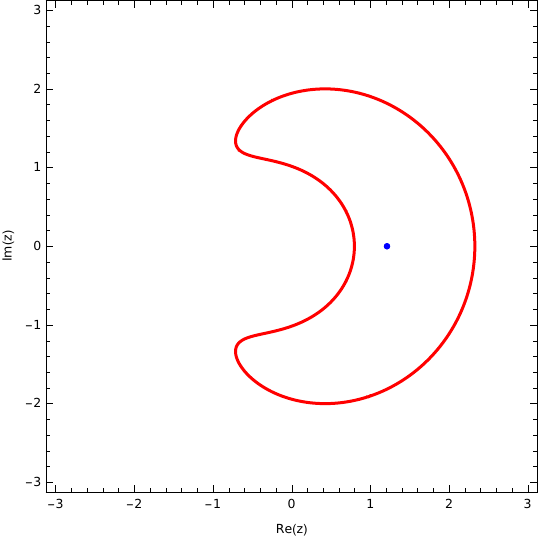}
		\caption{$a=8i/9$}
	\end{subfigure}
	\caption{Plots of $Z_0$ (blue) and $Z_1$ (red) for $v=-\frac{i}{2}\overline{z} + az^{-2}$ with different values of $a$.}
    \label{k=-2}
\end{figure}

\begin{figure}[htbp]
	\centering
	\begin{subfigure}{0.32\textwidth}
		\includegraphics[width=\linewidth]{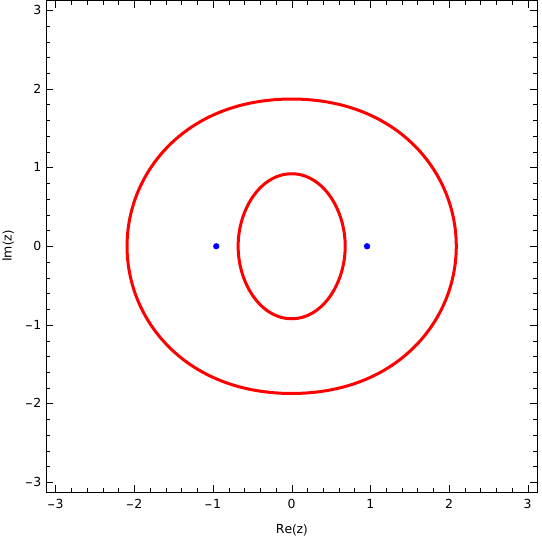}
		\caption{$a=27i/64$}
	\end{subfigure}
	\begin{subfigure}{0.32\textwidth}
		\includegraphics[width=\linewidth]{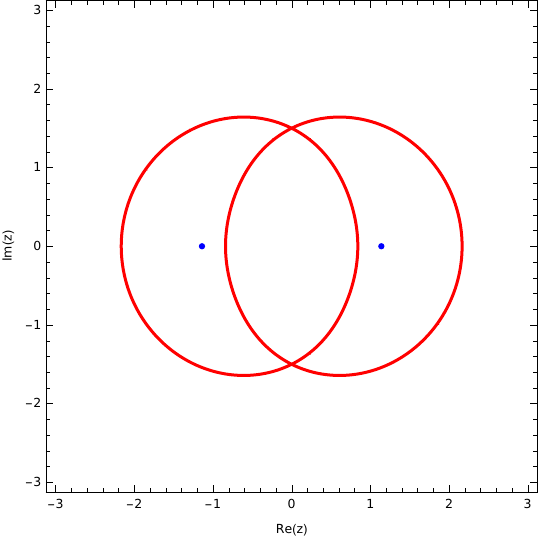}
		\caption{$a=27i/32$ ($=ia_*$)}
	\end{subfigure}
	\begin{subfigure}{0.32\textwidth}
		\includegraphics[width=\linewidth]{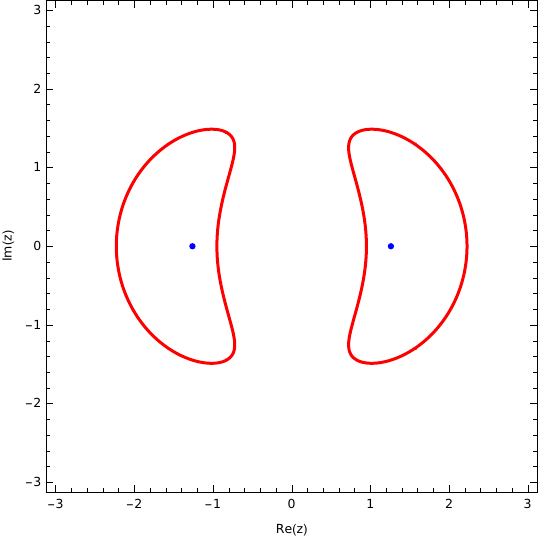}
		\caption{$a=81i/64$}
	\end{subfigure}
	\caption{Plots of $Z_0$ (blue) and $Z_1$ (red) for $v=-\frac{i}{2}\overline{z} + az^{-3}$ with different values of $a$.}
    \label{k=-3}
\end{figure}

\begin{figure}[htbp]
	\centering
	\begin{subfigure}{0.32\textwidth}
		\includegraphics[width=\linewidth]{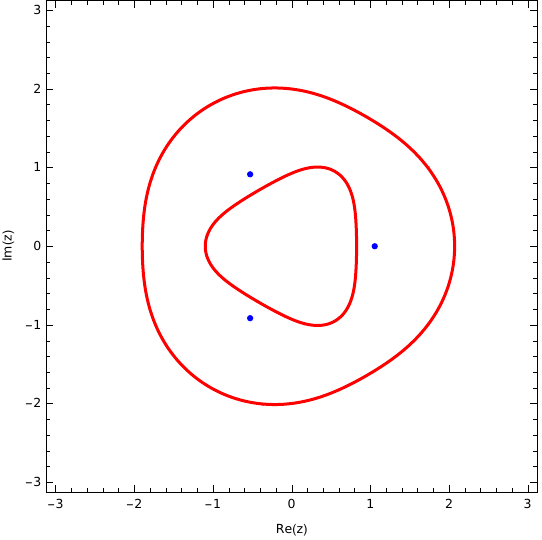}
		\caption{$a=2048i/3125$}
	\end{subfigure}
	\begin{subfigure}{0.32\textwidth}
		\includegraphics[width=\linewidth]{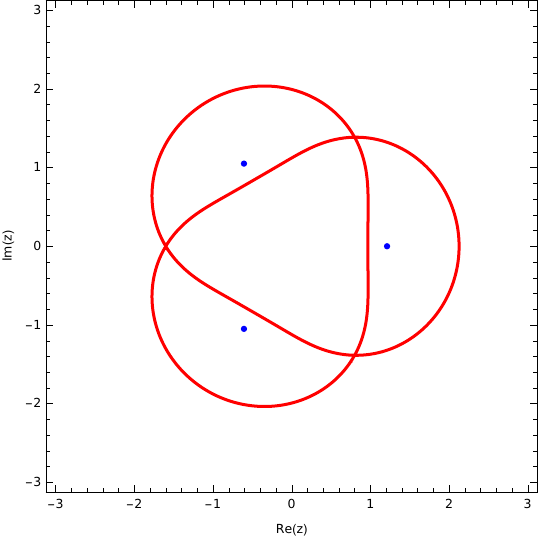}
		\caption{$a=4096i/3125$ ($=ia_*$)}
	\end{subfigure}
	\begin{subfigure}{0.32\textwidth}
		\includegraphics[width=\linewidth]{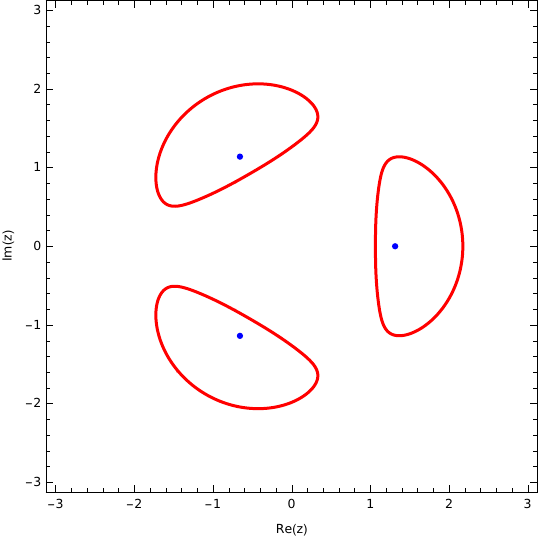}
		\caption{$a=6144i/3125$}
	\end{subfigure}
	\caption{Plots of $Z_0$ (blue) and $Z_1$ (red) for $v=-\frac{i}{2}\overline{z} + az^{-4}$ with different values of $a$.}
    \label{k=-4}
\end{figure}

\begin{figure}[htbp]
	\centering
	\begin{subfigure}{0.32\textwidth}
		\includegraphics[width=\linewidth]{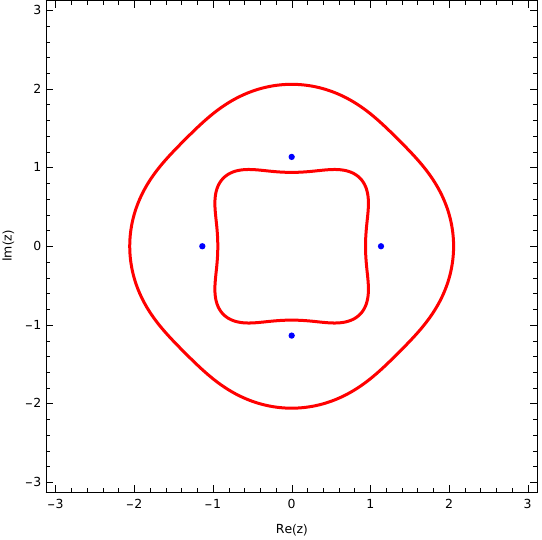}
		\caption{$a=3125i/2916$}
	\end{subfigure}
	\begin{subfigure}{0.32\textwidth}
		\includegraphics[width=\linewidth]{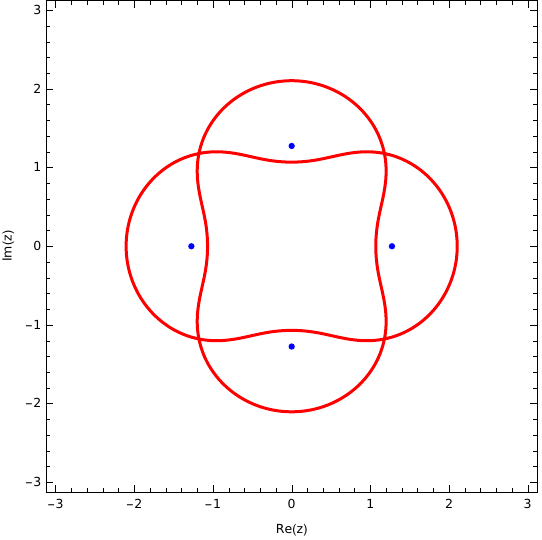}
		\caption{$a=3125i/1458$ ($=ia_*$)}
	\end{subfigure}
	\begin{subfigure}{0.32\textwidth}
		\includegraphics[width=\linewidth]{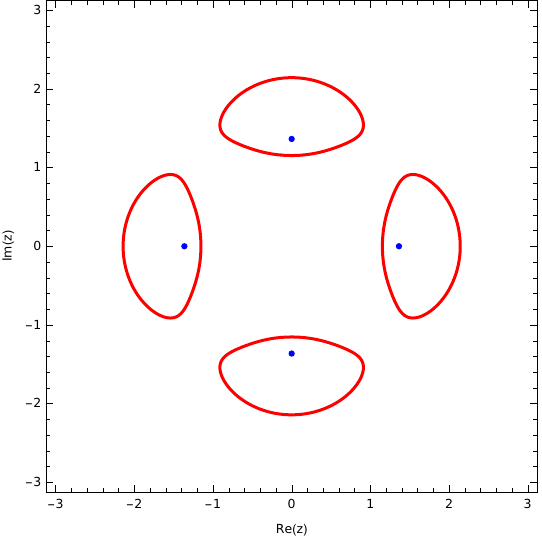}
		\caption{$a=46875i/972$}
	\end{subfigure}
	\caption{Plots of $Z_0$ (blue) and $Z_1$ (red) for $v=-\frac{i}{2}\overline{z} + az^{-5}$ with different values of $a$.}
    \label{k=-5}
\end{figure}

\FloatBarrier 
\bibliography{LL}

\end{document}